\documentclass[a4paper, 12pt]{article}
\usepackage{geometry}                
\usepackage{graphicx}
\usepackage{amssymb,amsmath,scalerel}
\usepackage{amsthm}
\usepackage[english]{babel}
\usepackage{thmtools, thm-restate}
\usepackage[hyphens]{url}
\usepackage[plainpages = false]{hyperref}
\usepackage[capitalize]{cleveref}
\usepackage[utf8]{inputenc}
\usepackage{enumerate}
\usepackage{xcolor}
\usepackage[backend = biber, style = numeric, giveninits]{biblatex}
\usepackage[hyphenbreaks]{breakurl}

\DeclareMathOperator*{\Times}{\scalerel*{\times}{\textstyle\sum}}
\DeclareMathOperator*{\Plus}{\scalerel*{+}{\textstyle\sum}}
\DeclareMathOperator{\Aut}{Aut}
\DeclareMathOperator{\ch}{ch}

\DeclareMathOperator{\End}{End}
\DeclareMathOperator{\Fix}{Fix}

\DeclareMathOperator{\Hom}{Hom}
\DeclareMathOperator{\Id}{Id}
\DeclareMathOperator{\im}{Im}
\DeclareMathOperator{\Ind}{Ind}

\DeclareMathOperator{\Irr}{Irr}

\DeclareMathOperator{\ord}{ord}

\DeclareMathOperator{\SpecR}{Spec_R}
\DeclareMathOperator{\Spec}{Spec}

\DeclareMathOperator{\Tr}{Tr}
\newcommand*{\ab}[1]{{#1}^{\mathrm{ab}}}
\newcommand*{\C}{\mathbb{C}}
\newcommand*{\Class}{\mathcal{C}}
\newcommand*{\Cunits}{\units{\C}}
\newcommand*{\dual}[1]{\widehat{#1}}
\newcommand*{\eg}{e.g.\ }

\newcommand*{\grpgen}[1]{\left\langle{#1}\right\rangle}

\newcommand*{\ie}{i.e.\ }

\newcommand*{\inv}[1]{{#1}^{-1}}
\newcommand*{\middlebar}{\ \middle | \ }

\newcommand*{\N}{\mathbb{N}}

\newcommand*{\Rinf}{R_{\infty}}

\newcommand*{\Reid}{\mathcal{R}}

\newcommand*{\size}[1]{\left| #1 \right|}

\newcommand*{\unitsb}[1]{\units{\left(#1\right)}}
\newcommand*{\units}[1]{{#1}^{\times}}
\newcommand*{\ZmodZ}[1]{\Z / #1 \Z}
\newcommand*{\ZnZ}{\ZmodZ{n}}

\newcommand*{\Z}{\mathbb{Z}}

\renewcommand{\phi}{\varphi}

\let\originalleft\left
\let\originalright\right
\renewcommand{\left}{\mathopen{}\mathclose\bgroup\originalleft}
\renewcommand{\right}{\aftergroup\egroup\originalright}

\declaretheorem[style=definition, name = Definition, numberwithin=subsection]{defin}

\declaretheorem[name = Theorem, sibling=defin]{theorem}
\declaretheorem[name = Lemma, sibling=defin]{lemma}
\declaretheorem[name = Proposition, sibling=defin]{prop}
\declaretheorem[name = Corollary, sibling=defin]{cor}
\declaretheorem[style=remark, name = Remark, numbered = no]{remark}

\declaretheorem[name = Theorem, numbered = no]{theorem*}

\numberwithin{equation}{section}

\crefname{prop}{Proposition}{Propositions}
\crefname{cor}{Corollary}{Corollaries}

\title{Reidemeister spectrum of split metacyclic groups}
\author{Pieter Senden\footnote{Researcher funded by FWO PhD-fellowship fundamental research (file number: 1112522N).}}

\begin{document}
\begin{center}
	\LARGE{The Reidemeister spectrum of split metacyclic groups}\\[.5em]
	\Large{Pieter Senden\footnote{Researcher funded by FWO PhD-fellowship fundamental research (file number: 1112522N).}}
\end{center}
\begin{abstract}
	Given a group $G$ and an automorphism $\varphi$ of $G$, two elements $x, y \in G$ are said to be $\varphi$-conjugate if $x = gy \varphi(g)^{-1}$ for some $g \in G$. The number of equivalence classes for this relation is the Reidemeister number $R(\varphi)$ of $\varphi$. The set $\{R(\psi) \mid \psi \in \mathrm{Aut}(G)\}$ is called the Reidemeister spectrum of $G$. We fully determine the Reidemeister spectrum of split metacyclic groups of the form $C_{n} \rtimes C_{p}$ where $p$ is a prime and the action is non-trivial.
\end{abstract}
\let\thefootnote\relax
\let\thefootnote\relax
\section{Introduction}
Let \(G\) be a group and \(\phi: G \to G\) be an automorphism. For \(x, y \in G\), we say that \(x\) and \(y\) are \emph{\(\phi\)-conjugate} if there exists a \(g \in G\) such that \(x = g y \inv{\phi(g)}\).	
	We define \(\Reid[\phi]\) to be the set of all \(\phi\)-equivalence classes and the \emph{{Reidemeister} number \(R(\phi)\) of \(\phi\)} as the cardinality of \(\Reid[\phi]\). Note that \(R(\phi) \in \N_0 \cup \{\infty\}\). Finally, we define the \emph{Reidemeister spectrum} to be
	\(
		\Spec_R(G) := \{ R(\phi) \mid \phi \in \Aut(G)\}.
	\)
	If \(\SpecR(G) = \{\infty\}\), we say that \(G\) has the \(\Rinf\)-property.

	There is also a topological Reidemeister number, which is used in Nielsen fixed-point theory to provide a bound on the number of fixed-point classes of a continuous self-map. Both numbers are strongly related, see \cite{Jiang83}. Other applications of twisted conjugacy appear in isogredience classes (see \eg \cite{FelshtynTroitsky15}) and representation theory (see \eg \cite{FelshtynLuchnikovTroitsky15,Springer06}).
	
	For many groups either partial or full information is known about their Reidemeister spectrum. It has either been proven that they have the \(\Rinf\)-property (\eg Baumslag-Solitar groups \cite{FelshtynGoncalves06}, Thompson's group \cite{BleakFelshtynGoncalves08}), or that they do not (\eg certain groups of exponential growth \cite{GoncalvesWong03}, free groups of infinite rank \cite{DekimpeGoncalves14}), and for some even the complete Reidemeister spectrum has been determined (\eg low-dimensional crystallographic groups \cite{DekimpeKaiserTertooy19}).
	
	In this article, we will only work with finite groups. Hence, the considered Reidemeister spectra will be finite sets of finite numbers and no such group can have the \(\Rinf\)-property. As far as the author knows, there is very little literature with results specifically concerning the Reidemeister spectrum of finite groups. One of the few results of such kind is due to A.\ Fel'shtyn, who proved that the Reidemeister number of an endomorphism \(\phi\) of a finite group equals the number of conjugacy classes that are mapped to itself by \(\phi\) (\cite[Theorem~14]{Felshtyn00}, cf.\ \cref{prop:ReidemeisterNumberEqualsNumberOfFixedConjugacyClasses}).

	The aim of this article is to determine explicit expressions for \(\SpecR(C_{n} \rtimes C_{p})\). Here, \(n, p\) are positive integers with \(p\) prime, \(C_{n}\) denotes the cyclic group of order \(n\) and the action of \(C_{p}\) on \(C_{n}\) is non-trivial; this family of groups contains among others all finite dihedral groups. We do this by counting the number of irreducible characters of the group that are fixed when composed with a given automorphism \(\phi\). This number of fixed irreducible characters turns out to be equal to \(R(\phi)\).
	
	This article consists of three parts. In the first part, we provide the necessary background and results regarding Reidemeister numbers and character theory. In a short second part, we recall the possibilities for the action of \(C_{p}\) on \(C_{n}\) and make some simplifications to reduce the number of semi-direct products we have to consider. The third part, finally, is devoted to determining the Reidemeister spectra of all remaining semi-direct products.
\section{Preliminaries}
\subsection{Reidemeister numbers}
In this section we recall the necessary tools for determining the Reidemeister spectrum and we also compute the Reidemeister spectrum of finite cyclic groups.
\begin{prop}[{See \eg \cite[Proposition~2.4]{Senden21}}]	\label{prop:ReidemeisterNumberDirectProductAutomorphismGroups}
	Let \(G_{1}, \ldots, G_{n}\) be groups. Consider an element \(\phi = (\phi_{1}, \ldots, \phi_{n})\) of \(\Aut(G_{1}) \times \ldots \times \Aut(G_{n}) \leq \Aut(G_{1} \times \ldots \times G_{n})\). Then \(R(\phi) = \prod_{i = 1}^{n} R(\phi_{i})\).
\end{prop}
	The following is well-known.
\begin{prop}	\label{prop:automorphismGroupDirectProduct}
	Let \(G = G_{1} \times \ldots \times G_{n}\) be a product of finite groups such that \(\gcd(\size{G_{i}}, \size{G_{j}}) = 1\) for \(i \ne j\). Then
	\[
		\Aut(G) \cong \Times_{i = 1}^{n} \Aut(G_{i}).
	\]
\end{prop}
\begin{defin}
	Let \(A_{1}, \ldots, A_{n}\) be sets of natural numbers. We define
	\[
		A_{1} \cdot \ldots \cdot A_{n} := \prod_{i = 1}^{n} A_{i} := \{a_{1} \ldots a_{n} \mid \forall i \in \{1, \ldots, n\}: a_{i} \in A_{i}\}.
	\]
	and
	\[
		A_{1} + \ldots + A_{n} := \Plus_{i = 1}^{n} A_{i} := \{a_{1} + \ldots + a_{n} \mid \forall i \in \{1, \ldots, n\}: a_{i} \in A_{i}\}
	\]
\end{defin}

\begin{cor}	\label{cor:SpecDirectProductCoprime}
	Let \(G = G_{1} \times \ldots \times G_{n}\) be a product of finite groups such that \(\gcd(\size{G_{i}}, \size{G_{j}}) = 1\) for \(i \ne j\). Then
	\[
		\SpecR(G) = \prod_{i = 1}^{n} \SpecR(G_{i}).
	\]
\end{cor}
\begin{lemma}[{See \eg \cite[Corollary~2.5]{FelshtynTroitsky07}}]	\label{lem:ReidemeisterNumberCompositionInnerAutomorphism}
	Let \(G\) be a group, \(\phi \in \Aut(G)\) and \(g \in G\). Denote by \(\tau_{g}\) the inner automorphism associated to \(g\). Then \(R(\tau_{g} \circ \phi) = R(\phi)\).
\end{lemma}
\begin{prop}[{See \eg \cite[Theorem 14]{Felshtyn00}}]	\label{prop:ReidemeisterNumberEqualsNumberOfFixedConjugacyClasses}
	Let \(G\) be a finite group and \(\phi \in \End(G)\) an endomorphism. Denote by \(\Class := \{[x] \mid x \in G\}\) the set of conjugacy classes of \(G\). Then
	\[
		R(\phi) = \size{\{[x] \in \Class \mid [\phi(x)] = [x]\}}.
	\]
	In particular, if \(G\) is finite abelian, then \(R(\phi) = \size{\Fix(\phi)}\).
\end{prop}

\begin{lemma}	\label{lem:FixedPointsCyclicGroup}
	Let \(n \geq 2\) and let \(\phi \in \End(C_{n})\) be given by \(\phi(x) = x^{\gamma}\), where \(x\) is a generator and \(\gamma \in \Z\). Put \(d = \gcd(\gamma - 1, n)\). Then \(\Fix(\phi) = \grpgen{x^{\frac{n}{d}}}\).
	
	In particular, \(R(\phi) = \gcd(\gamma - 1, n)\).
\end{lemma}
\begin{proof}
	Note that \(\gamma\) is only defined modulo \(n\), hence we have to check that \(\gcd(\gamma - 1, n) = \gcd(\gamma' - 1, n)\) if \(\gamma \equiv \gamma' \bmod n\). However, writing \(\gamma' = \gamma + \alpha n\) for some \(\alpha \in \Z\), we see that \(\gcd(\gamma' - 1, n) = \gcd(\gamma + \alpha n - 1, n) = \gcd(\gamma - 1, n)\).

	We now determine the fixed points of \(\phi\). We have that \(\phi(x^{i}) = x^{i}\) if and only if \(i \cdot (\gamma - 1) \equiv 0 \bmod n\). Writing \(d = \gcd(\gamma - 1, n)\), we see that \(\frac{\gamma - 1}{d}\) is invertible modulo \(n\), hence \(i \cdot (\gamma - 1) \equiv 0 \bmod n\) if and only if \(i \cdot d \equiv 0 \bmod n\). Thus, for \(i \cdot d \equiv 0 \bmod n\) to hold, \(i\) must be a multiple of \(\frac{n}{d}\). Hence, \(\Fix(\phi) = \grpgen{x^{\frac{n}{d}}}\).
	
	As we work in an abelian group, \(R(\phi) = \size{\Fix(\phi)}\), therefore, \(R(\phi) = d = \gcd(\gamma - 1, n)\).
\end{proof}
\begin{prop}	\label{prop:SpecRCyclicGroups}
	Let \(n \geq 2\). Then
	\[
		\SpecR(C_{n}) = \begin{cases}
			\{d \mid d \text{ divides } n, d \geq 1\}	&	\mbox{if \(n\) is odd,}	\\
			\{d \mid d \text{ divides } n, d \geq 1, d \equiv 0 \bmod 2\}	&	\mbox{if \(n\) is even.}
		\end{cases}
	\]
\end{prop}
\begin{proof}
	Note that \(\Aut(C_{n}) \cong \unitsb{\ZnZ}\). Let \(\phi: C_{n} \to C_{n}: x \mapsto x^{\gamma}\), where \(\gamma \in \Z\), be an automorphism. Then \(\gcd(\gamma, n) = 1\). As \(R(\phi) = \gcd(\gamma - 1, n)\), we see that this number is even if \(n\) is even, since then \(\gamma\) must be odd. This proves the \(\subseteq\)-inclusion.
	
	For the converse, suppose that \(d\) is a divisor of \(n\) satisfying the necessary conditions. We have to find a \(\gamma \in \Z\) coprime with \(n\) such that \(\gcd(\gamma - 1, n) = d\). Let \(n = \prod_{i = 1}^{r} p_{i}^{e_{i}}\) be the prime factorisation of \(n\) with all \(e_{i} \geq 1\) and all \(p_{i}\) distinct primes. Write \(d = \prod_{i = 1}^{r} p_{i}^{f_{i}}\), where \(0 \leq e_{i} \leq f_{i}\), and consider the system of congruences
	\[
		x \equiv 1 + p_{i}^{f_{i}} \bmod p_{i}^{e_{i}}
	\]
	for all \(i\) with \(f_{i} \geq 1\). Let \(q\) be the product of all \(p_{i}^{e_{i}}\)'s for which \(f_{i} = 0\) and add to the system the congruence
	\[
		x \equiv -1 \bmod q.
	\]
	Note that \(q\) is odd, as \(d\) is even if \(n\) is even. The Chinese Remainder Theorem yields a solution \(\gamma \in \Z\) to this system of congruence, and \(\gamma\) is uniquely determined modulo \(n\), as the product of the moduli is precisely \(n\).
	
    	It is clear that \(\gcd(\gamma - 1, p_{i}^{e_{i}}) = p_{i}^{f_{i}}\) for all \(i\) with \(f_{i} \geq 1\). If \(f_{i} = 0\), then \(p_{i}\) is odd as \(2\) divides \(d\) if \(n\) is even. Furthermore, \(p_{i}\) divides \(\gamma + 1\) by the last congruence, implying that \(p_{i}\) does not divide \(\gamma - 1\) as \(p_{i}\) is odd. This proves that \(\gcd(\gamma - 1, n) = d\).
	
	Finally, if \(p_{i}\) divides \(\gamma\), then either \(0 \equiv 1 \bmod p_{i}\) or \(0 \equiv -1 \bmod p_{i}\), both of which are impossible. Hence, \(\gcd(\gamma, n) = 1\).
\end{proof}
	The following can easily be proven by using a generator of the cyclic group.
\begin{lemma}	\label{lem:intersectionSubgroupsCyclicGroup}
	Let \(C\) be a finite cyclic group and let \(H_{1}, H_{2}\) be two subgroups of order \(h_{1}, h_{2}\), respectively. Then
	\[
		|H_{1} \cap H_{2}| = \gcd(h_{1}, h_{2}).
	\]
\end{lemma}

\subsection{Character theory}
As mentioned in the introduction, we will determine Reidemeister numbers by counting fixed irreducible characters. In this section, we provide a precise statement and proof of this result, together with other necessary tools regarding irreducible characters.

Let \(G\) be a finite group. The \emph{dual group} is the group \(\dual{G} := \Hom(G, \Cunits)\) of group homomorphisms from \(G\) to \(\Cunits\). Note that \(\dual{G}\) is abelian. For two groups \(G, H\) and a homomorphism \(\phi: G \to H\), we define the group homomorphism
	\[
		\dual{\phi}: \dual{H} \to \dual{G}: \chi \mapsto \chi \circ \phi.
	\]
	It is well-known (see \eg \cite[Theorem~3.11]{ConradChar}) that if \(A\) is a finite abelian group, then \(\dual{A}\) is the group of irreducible characters of \(A\) and it is isomorphic to \(A\). Using this, one can prove more generally that the dual group \(\dual{G}\) of a finite group \(G\) is the group of all \(1\)-dimensional irreducible characters of \(G\) and is isomorphic to \(\dual{G / \gamma_{2}(G)} \cong G / \gamma_{2}(G)\) via the dual of the projection \(\pi: G \to G / \gamma_{2}(G)\).
\begin{prop}
	Let \(G\) be a finite group and let \(X := \{\chi_{1}, \ldots, \chi_{c}\}\) be the irreducible characters of \(G\). Let \(\phi \in \Aut(G)\). Then
	\[
		\Phi: X \to X: \chi \mapsto \chi \circ \phi
	\]
	is a well-defined bijection and \(R(\phi) = \size{\Fix(\Phi)}\)
\end{prop}
\begin{proof}
	Since \(\phi\) is an automorphism, \(\chi \circ \phi\) is an irreducible character whenever \(\chi\) is. Hence, \(\Phi\) is well-defined, and clearly, it is bijective. Let \(\Class(G)\) be the vector space of (complex valued) class functions on \(G\) and consider the linear map
	\[
		\Psi: \Class(G) \to \Class(G): f \mapsto f \circ \phi.
	\]
	We can consider two bases of \(\Class(G)\): the set \(X\) and the set \(\Delta := \{\Delta_{[x]} \mid [x] \in \Class\}\) of maps
	\[
		\Delta_{[x]}: G \to \C: g \mapsto \begin{cases}
			1	&	\mbox{if } g \in [x]	\\
			0	&	\mbox{if } g \notin [x],
		\end{cases}
	\]
	where \(\Class\) is the set of conjugacy classes of \(G\). The matrix representation of \(\Psi\) with respect to \(\Delta\) is a permutation matrix. There is a \(1\) on the diagonal if and only if the corresponding conjugacy class is fixed by \(\phi\). As \(R(\phi)\) is the number of fixed conjugacy classes under \(\phi\) by \cref{prop:ReidemeisterNumberEqualsNumberOfFixedConjugacyClasses}, we have that \(R(\phi) = \Tr(\Psi)\) (cf.\ \cite[Theorem~15]{Felshtyn00}). Similarly, the matrix representation of \(\Psi\) with respect to \(X\) is also a permutation matrix, since \(\Phi\) is a well-defined bijection. The trace of this matrix is precisely the number of \(1\)'s on the diagonal, which corresponds to the number of fixed points of \(\Phi\).
\end{proof}
\begin{defin}
	Let \(G\) be a finite group. For each divisor \(d\) of \(|G|\), we define \(\Irr_{d}(G)\) to be the set of all irreducible \(d\)-dimensional characters. We denote its cardinality by \(\ch_{d}(G)\). Given an automorphism \(\phi\) of \(G\), we write \(\ch_{d, \phi}(G)\) for the number of characters in \(\Irr_{d}(G)\) fixed by \(\phi\).
	
	If \(G\) is clear from the context, we simply write \(\ch_{d}\) and \(\ch_{d, \phi}\).
\end{defin}
\begin{lemma}	\label{lem:FixedCharactersDimension1}
	Let \(G\) be a finite group and \(\phi \in \Aut(G)\). Then \(\ch_{1, \phi} = R(\ab{\phi})\), where \(\ab{\phi}\) is the induced automorphism on \(G / \gamma_{2}(G)\).
\end{lemma}
\begin{proof}
	As mentioned earlier, we know that \(\Irr_{1}(G) = \dual{G}\) is isomorphic to \(\dual{\frac{G}{\gamma_{2}(G)}}\) via \(\dual{\pi}\), the dual of the canonical projection \(G \to G / \gamma_{2}(G)\). The map \(\dual{\phi}\) on \(\Irr_{1}(G)\) is thus linked to \(\dual{\ab{\phi}}\) on \(\dual{\frac{G}{\gamma_{2}(G)}}\) via
	\[
		\dual{\phi} = \dual{\pi} \circ \dual{\ab{\phi}} \circ \inv{\dual{\pi}}.
	\]
	Their number of fixed points will therefore be the same. Therefore, we may assume that \(G\) is abelian, since we can work with \(\ab{\phi}\).
	
	So, suppose that \(G\) is abelian and let \(\phi\) be an automorphism of \(G\). We know that \(R(\phi) = \size{\Fix(\phi)}\) due to \cref{prop:ReidemeisterNumberEqualsNumberOfFixedConjugacyClasses} and that, by definition, \(\ch_{1, \phi} = \size{\Fix(\dual{\phi})}\). Therefore, we need to prove that \(\size{\Fix(\phi)} = \size{\Fix(\dual{\phi})}\).
	
	For a subgroup \(H\) of \(G\), we define (based on \cite[Exercise~7, p.~11]{ConradChar})
	\[
		H^{\bot} := \{\chi \in \dual{G} \mid \forall h \in H: \chi(h) = 1\}.
	\]
	We claim that \(\size{H^{\bot}} = [G : H]\) and that
	\[
		\ker(\dual{\psi}) = \im(\psi)^{\bot}
	\]
	for all \(\psi \in \End(G)\).
	
	For the first claim, let \(\pi_{H}: G \to G / H\) be the natural projection. Then the map
	\[
		F : \dual{G / H} \to H^{\bot}: \chi \mapsto \chi \circ \pi_{H}
	\]
	is bijective. Indeed, it is clearly injective and if \(\tilde{\chi}\) is a homomorphism in \(H^{\bot}\), then it induces a homomorphism \(\chi\) of \(G / H\), as it vanishes on \(H\). Then \(\tilde{\chi} = \chi \circ \pi_{H}\), proving that \(F\) is surjective. Since \(\dual{G / H} \cong G / H\), they have the same size, namely \([G : H]\), implying that \(\size{H^{\bot}} = [G : H]\).
	
	For the second claim, let \(\psi \in \End(G)\). Then \(\chi \circ \psi\) is the trivial homomorphism if and only if \(\chi(\psi(g)) = 1\) for all \(g \in G\), therefore, if and only if \(\chi(h) = 1\) for all \(h \in \im \psi\).
	
	Now, consider the endomorphism \(\phi - \Id_{G}\) (recall that \(G\) was assumed to be abelian). Then both results combined say that
	\begin{equation*}	\label{eq:nbFixedPointsDualMap}
		\size{\ker(\dual{\phi - \Id_{G}})} = [G : \im(\phi - \Id_{G})] = \size{\ker(\phi - \Id_{G})} = \size{\Fix(\phi)} = R(\phi).
	\end{equation*}
	Finally, note that for \(\chi \in \dual{G}\), we have
	\begin{align*}
		\dual{\phi - \Id_{G}}(\chi) = 0	& \iff \forall g \in G: \chi(\phi(g)) = \chi(g)	\\
								& \iff \chi \circ \phi = \chi	\\
								& \iff \chi \in \Fix(\dual{\phi}).
	\end{align*}
	Therefore, \(\size{\ker(\dual{\phi - \Id_{G}})} = \size{\Fix(\dual{\phi})}\). Consequently,
	\[
		\ch_{1, \phi} = \size{\Fix(\dual{\phi})} = \size{\ker(\dual{\phi - \Id_{G}})} = \size{\Fix(\phi)}= R(\phi),
	\]
	finishing the proof.
\end{proof}
For future reference, we state the equality \(\size{\Fix(\phi)} = \size{\Fix(\dual{\phi})}\) as a separate result. 
\begin{cor}	\label{cor:nbFixedPointsDualMap}
	Let \(G\) be a finite abelian group and \(\phi \in \Aut(G)\). Then \(\size{\Fix(\phi)} = \size{\Fix(\dual{\phi})}\).
\end{cor}

\subsection{Characters of \(A \rtimes C_{p}\)}	\label{subsec:CharactersSemidirectProduct}
In order to apply the technique of counting fixed characters, we need an expression for the characters on semi-direct products of the form \(A \rtimes C_{p}\), with \(A\) abelian. We recall a more general construction, details of which can be found in \cite[\S8.2]{Serre77}.

Suppose \(G = A \rtimes H\) is a semi-direct product of a finite abelian group \(A\) and another finite group \(H\). Then \(H\) acts on \(\dual{A}\) via conjugation, \ie
\[
	(h \cdot \chi)(a) := \chi(a^{h})
\]
for \(h \in H, a \in A\) and \(\chi \in \dual{A}\).
Let \(\{\chi_{1}, \ldots, \chi_{m}\}\) be a system of representatives of the orbits in \(\dual{A}\) of this action. For \(i \in \{1, \ldots, m\}\), let \(H_{i}\) be the stabiliser of \(\chi_{i}\) and put \(G_{i} := A \rtimes H_{i} \leq G\). We can extend \(\chi_{i}\) to \(G_{i}\) by putting \(\chi_{i}(ah) := \chi_{i}(a)\) for \(a \in A, h \in H_{i}\). Since \(h \cdot \chi_{i} = \chi_{i}\) for all \(h \in H_{i}\), we see that \(\chi_{i}\) is a \(1\)-dimensional character of \(G_{i}\). Next, let \(\rho\) be an irreducible representation of \(H_{i}\). Then composing \(\rho\) with the canonical projection from \(G_{i}\) to \(H_{i}\) yields an irreducible representation \(\tilde{\rho}\) of \(G_{i}\). Finally, put
\[
	\theta_{i, \rho} := \Ind_{G_{i}}^{G} (\chi_{i} \otimes \tilde{\rho}),
\]
the induced representation of \(\chi_{i} \otimes \tilde{\rho}\). Note that \(\dim(\theta_{i, \rho}) = [G : G_{i}] \dim \rho = [H : H_{i}] \dim \rho\). 
\begin{theorem}[{See \eg \cite[Proposition 25 \& Theorem 12]{Serre77}}]	\label{theo:RepresentationsSemidirectProduct}
	Let \(G = A \rtimes H\) be finite with \(A\) abelian. Then the following hold:
	\begin{enumerate}[1)]
		\item Each irreducible representation of \(G\) is isomorphic to a representation \(\theta_{i, \rho}\) of the form constructed above.
		\item Two representations \(\theta_{i, \rho}\) and \(\theta_{j, \sigma}\) are isomorphic if and only if \(i = j\) and \(\rho\) and \(\sigma\) are isomorphic.
		\item The character \(\overline{\chi}_{i, \rho}\) of \(\theta_{i, \rho}\) is given by
	\[
		\overline{\chi}_{i, \rho}(g) = \sum\limits_{\substack{h \in R \\ g^{h} \in G_{i}}} (\chi_{i} \otimes \chi_{\tilde{\rho}})(g^{h}) = \frac{1}{\size{G_{i}}} \sum\limits_{\substack{h \in G \\ g^{h} \in G_{i}}} (\chi_{i} \otimes \chi_{\tilde{\rho}})(g^{h})
	\]
	where \(R\) is a set of representatives of \(G / G_{i}\).
	\end{enumerate}
\end{theorem}

We now consider the case where \(H = C_{p}\) and \(A\) is any (non-trivial) finite abelian group. Throughout the remainder of the section, we put \(G = A \rtimes_{\alpha} C_{p}\), where \(\alpha: C_{p} \to \Aut(A)\), and we fix a generator \(y\) of \(C_{p}\).

\begin{lemma}	\label{lem:DimensionCharactersACp}
	Each irreducible character of \(G\) has dimension \(1\) or \(p\).
\end{lemma}
\begin{proof}
Let \(\chi \in \dual{A}\) be a character. Its stabiliser for the action of \(C_{p}\) is either \(1\) or \(C_{p}\). If its stabiliser is \(C_{p}\), then the representation \(\rho\) in the construction above is \(1\)-dimensional. Therefore, the induced character has dimension \([G : G] \dim \rho = 1\).

If the stabiliser equals \(1\), then \(\rho\) and \(\tilde{\rho}\) are both trivial representations. Thus, the induced character has dimension \([G : A] \dim \rho = [C_{p} : 1] = p\).
\end{proof}

If \(\chi \in \dual{A}\) is a character inducing a \(p\)-dimensional character on \(G\), we denote this induced character by \(\overline{\chi}\). We can simplify the explicit expression for the induced characters in \(\Irr_{p}(G)\).

\begin{lemma}	\label{lem:pCharacterACp}
	Let \(\chi\) be a character on \(A\) inducing a \(p\)-dimensional character \(\overline{\chi}\) on \(G\). Then
	\[
		\overline{\chi}(g) = \begin{cases}
			0	&	\mbox{if } g \not \in A	\\
			\sum\limits_{i = 0}^{p - 1} \chi(\alpha(y)^{i}(g))	&	\mbox{if } g \in A
		\end{cases}
	\]
\end{lemma}
\begin{proof}
Since \(\overline{\chi}\) is \(p\)-dimensional, the stabiliser of \(\chi\) under the action of \(C_{p}\) is trivial. Therefore, the \(G_{i}\) from the construction equals \(A\) and the group \(C_{p}\) can be used as a set of representatives of \(G / A\). Recall that this also implies that the representation \(\rho\) in the construction is the trivial representation. Consequently, if \(g \in A\), then \((\chi \otimes \tilde{\rho})(g^{y^{i}}) = \chi(\alpha(y)^{i}(g))\). Since \(A\) is normal in \(G\), either all conjugates of \(g\) lie in \(A\) or none of them do. Hence, if \(g \in A\),
\[
	\overline{\chi}(g) = \sum\limits_{i = 0}^{p - 1}\chi(\alpha(y)^{i}(g))
\]
and if \(g \notin A\), then \(\overline{\chi}(g) = 0\).
\end{proof}

\cref{lem:FixedCharactersDimension1} allows us to determine the number of characters in \(\Irr_{1}(G)\) fixed by a given automorphism. For those in \(\Irr_{p}(G)\), we can use the following criterion.

\begin{lemma}	\label{lem:FixedCharactersDimensionp}
	Let \(\overline{\chi}\) be a character in \(\Irr_{p}(G)\) and \(\phi \in \Aut(G)\) such that \(\phi(A) = A\). Denote by \(\phi'\) the induced automorphism on \(A\). Then \(\overline{\chi}\) is fixed by \(\phi\) if and only if there is an \(i \in \{0, \ldots, p - 1\}\) such that \(\chi \circ \phi' = \chi \circ \alpha(y)^{i}\).
\end{lemma}
\begin{proof}
	As \(\phi(A) = A\) and \(\phi\) is bijective, also \(\phi(G \setminus A) = G \setminus A\), implying that \(\overline{\chi}(\phi(g)) = 0\) if \(g \notin A\). Thus, \(\overline{\chi} \circ \phi = \overline{\chi}\) if and only if
	\[
		\sum\limits_{i = 0}^{p - 1} \chi(\alpha(y)^{i}(g)) = \sum\limits_{i = 0}^{p - 1} \chi(\alpha(y)^{i}(\phi'(g)))
	\]
	for all \(g \in A\). We can rewrite this as
	\[
		\sum\limits_{i = 0}^{p - 1} (\chi \circ \alpha(y)^{i})(g) = \sum\limits_{i = 0}^{p - 1} (\chi \circ \alpha(y)^{i} \circ \phi')(g),
	\]
	and as it has to hold for all \(g \in A\), we get the following equality of class functions:
	\begin{equation}	\label{eq:equalityOfCharacterSums}
		\sum\limits_{i = 0}^{p - 1} \chi \circ \alpha(y)^{i} = \sum\limits_{i = 0}^{p - 1} \chi \circ \alpha(y)^{i} \circ \phi'.
	\end{equation}
	As \(\alpha(y)\) and \(\phi'\) are automorphisms, each term in either sum is an irreducible character of \(A\). Moreover, no two terms in the same sum are equal. Indeed, if \(\chi \circ \alpha(y)^{i} = \chi \circ \alpha(y)^{j}\) for \(0 \leq i, j \leq p - 1\), then \(\chi \circ \alpha(y)^{i - j} = \chi\). This implies that \(y^{i - j}\) lies in the stabiliser of \(\chi\) under the action of \(C_{p}\). As this stabiliser is trivial, this means that \(i \equiv j \bmod p\), and thus \(i = j\). Since \(\phi'\) is an automorphism, we can apply the same argument to the terms \(\chi \circ \alpha(y)^{i} \circ \phi'\).
	
	Thus, each side in \eqref{eq:equalityOfCharacterSums} is a sum of distinct irreducible characters of \(A\). As the irreducible characters on \(A\) form a basis of the vector space of class functions on \(A\), equality in \eqref{eq:equalityOfCharacterSums} holds if and only if
	\begin{equation}	\label{eq:equalityOfCharacterSets}
		\{\chi \circ \alpha(y)^{i} \mid 0 \leq i \leq p - 1 \} = \{\chi \circ \alpha(y)^{i} \circ \phi' \mid 0 \leq i \leq p - 1 \}.
	\end{equation}
	This equality of sets immediately implies that \(\chi \circ \phi' = \chi \circ \alpha(y)^{i}\) for some \(i \in \{0, \ldots, p - 1\}\). Conversely, suppose that \(\chi \circ \phi' = \chi \circ \alpha(y)^{i}\) for some \(i \in \{0, \ldots, p - 1\}\). Let \(g \in A\). As \(\phi\) is a group homomorphism, it preserves the relation \(g^{y} = \alpha(y)(g)\). Write \(\phi(y) = a y^{k}\) for some \(a \in A, k \in \Z\). Then
	\[
		\alpha(y)^{k}(\phi(g)) = \phi(g)^{y^{k}} = \phi(g)^{ay^{k}} = \phi(g)^{\phi(y)} = \phi(g^{y}) = \phi(\alpha(y)(g)).
	\]
	As this holds for all \(g \in A\), we get the equality \(\alpha(y)^{k} \circ \phi' = \phi' \circ \alpha(y)\) of automorphisms on \(A\). Note that \(k \not\equiv 0 \bmod p\), as \(\phi\) is an automorphism. From this equality we find that
	\[
		\alpha(y)^{jk} \circ \phi' = \phi' \circ \alpha(y)^{j}
	\]
	for all \(j \in \Z\).
	
	Using this, we proceed in proving equality in \eqref{eq:equalityOfCharacterSets}. Let \(l\) be the multiplicative inverse of \(k\) modulo \(p\). For \(j \in \{0, \ldots, p - 1\}\), we then have
	\[
		\chi \circ \alpha(y)^{j} \circ \phi' = \chi \circ \phi' \circ \alpha(y)^{jl} = \chi \circ \alpha(y)^{i + jl}.
	\]
	As \(l\) is invertible modulo \(p\), the set \(\{i + jl \mid j \in \{0, \ldots, p - 1\}\}\) forms a complete set of representatives modulo \(p\). Consequently, we find that equality in \eqref{eq:equalityOfCharacterSets} holds.
\end{proof}

With this characterisation we can determine an explicit formula for \(\ch_{p, \phi}(G)\) given an automorphism \(\phi \in \Aut(G)\).

\begin{prop}	\label{prop:generalExpressionchpphi}
	Let \(\phi \in \Aut(G)\) be such that \(\phi(A) = A\) and denote by \(\phi'\) the induced automorphism on \(A\). Then
	\[
		\ch_{p, \phi}(G) = \frac{1}{p} \left(\sum_{i = 0}^{p - 1} \size{\Fix(\phi' \circ \alpha(y)^{i})}\right) - \size{\Fix(\dual{\phi'}) \cap \Fix(\dual{\alpha(y)})}
	\]
\end{prop}
\begin{proof}
	Let \(\phi \in \Aut(G)\) with \(\phi(A) = A\) be fixed. We use \cref{lem:FixedCharactersDimensionp} to determine \(\ch_{p, \phi}\). Define for \(i \in \{0, \ldots, p - 1\}\) the set
	\[
		F_{\phi, i} = \{\chi \in \dual{A} \mid \chi \circ \phi' = \chi \circ \alpha(y)^{i}\}.
	\]
	By the aforementioned lemma, if \(\chi \in \dual{A}\) induces a \(p\)-dimensional character of \(G\), then this induced character is fixed by \(\phi\) if and only if \(\chi\) lies in one of the \(F_{\phi, i}\). However, not all characters in \(F_{\phi, i}\) will induce a \(p\)-dimensional one on \(G\). Therefore, we have to determine which do.
	
	We start by computing the size of the union of all \(F_{\phi, i}\)'s using the inclusion-exclusion principle. For \(i \ne j \in \{0, \ldots, p - 1\}\), consider \(F_{\phi, i} \cap F_{\phi, j}\). A character \(\chi \in \dual{A}\) lies in this intersection if and only if
	\[
		\chi \circ \phi' = \chi \circ \alpha(y)^{i} = \chi \circ \alpha(y)^{j}.
	\]
	The second equality is equivalent with \(\chi \in \Fix(\dual{\alpha(y)^{j - i}})\). As \(j - i \not \equiv 0 \bmod p\) and \(\dual{\alpha(y)}\) has order \(p\), it holds that \(\Fix(\dual{\alpha(y)^{j - i}}) = \Fix(\dual{\alpha(y)})\). From this it also follows that \(\chi \circ \phi' = \chi\), meaning that \(\chi \in F_{\phi, 0}\). Therefore, \(\chi \in \Fix(\dual{\alpha(y)}) \cap F_{\phi, 0}\).  Clearly, the converse holds as well, \ie \(\chi \in \Fix(\dual{\alpha(y)}) \cap F_{\phi, 0}\) implies \(\chi \circ \phi' = \chi \circ \alpha(y)^{i} = \chi \circ \alpha(y)^{j}\). We conclude that
	\begin{equation}	\label{eq:intersectionFphiiFphij}
		F_{\phi, i} \cap F_{\phi, j} = F_{\phi, 0} \cap \Fix(\dual{\alpha(y)}).
	\end{equation}
	Consequently, any intersection of at least two \(F_{\phi, i}\)'s with distinct indices equals \(F_{\phi, 0} \cap \Fix(\dual{\alpha(y)})\). Thus, the inclusion-exclusion principle yields
	\begin{align*}
		\size{\bigcup_{i = 0}^{p - 1} F_{\phi, i}}	&= \sum_{i = 0}^{p - 1} \size{F_{\phi, i}} + \sum_{i = 2}^{p} (-1)^{i + 1} \binom{p}{i} \size{F_{\phi, 0} \cap \Fix(\dual{\alpha(y)})}	\\
										&= \sum_{i = 0}^{p - 1} \size{F_{\phi, i}} - \size{F_{\phi, 0} \cap \Fix(\dual{\alpha(y)})} \sum_{i = 2}^{p} (-1)^{i} \binom{p}{i}	\\
										&= \sum_{i = 0}^{p - 1} \size{F_{\phi, i}} - \size{F_{\phi, 0} \cap \Fix(\dual{\alpha(y)})} \left(-\binom{p}{0} + \binom{p}{1}\right)	\\
										&= \sum_{i = 0}^{p - 1} \size{F_{\phi, i}} - (p - 1)\size{F_{\phi, 0} \cap \Fix(\dual{\alpha(y)})}.
	\end{align*}
We now determine which characters in the union induce a \(p\)-dimensional one on \(G\). Recall that \(\chi \in \dual{A}\) induces \(\overline{\chi} \in \Irr_{p}(G)\) if and only if its stabiliser under the action of \(C_{p}\) on \(\dual{A}\) is trivial. As \((y^{i} \cdot \chi)(a) = \chi(a^{y^{i}}) = \chi(\alpha(y)^{i}(a))\) for all \(a \in A\), the stabiliser of \(\chi\) is trivial if and only if \(\chi \circ \alpha(y)^{i} \ne \chi\) for at least one \(i \in \{1, \ldots, p - 1\}\). As \(\alpha(y)\) has prime order, this is equivalent with \(\chi \circ \alpha(y) \ne \chi\), \ie \(\chi \not \in \Fix(\dual{\alpha(y)})\).

Therefore, we must subtract \(\size{\Fix(\dual{\alpha(y)}) \cap \bigcup_{i = 0}^{p - 1} F_{\phi, i}}\). By a similar argument as for \eqref{eq:intersectionFphiiFphij} we find that
\[
	F_{\phi, i} \cap \Fix(\dual{\alpha(y)}) = F_{\phi, 0} \cap \Fix(\dual{\alpha(y)})
\]
for all \(i \in \{0, \ldots, p - 1\}\). Hence, we have to subtract \(\size{F_{\phi, 0} \cap \Fix(\dual{\alpha(y)})}\) and obtain
\[
	\sum_{i = 0}^{p - 1} \size{F_{\phi, i}} - p\size{F_{\phi, 0} \cap \Fix(\dual{\alpha(y)})}.
\]
We have to take one last thing into account. If \(\chi \in \dual{A}\) induces a \(p\)-dimensional character on \(G\) fixed by \(\phi\), then so does \(\chi \circ \alpha(y)^{i}\) for each \(i \in \{0, \ldots, p - 1\}\). However, all these induced characters are equal. Conversely, by the second item of \cref{theo:RepresentationsSemidirectProduct}, this is the only way in which two characters of \(A\) can induce the same \(p\)-dimensional character of \(G\). In other words, each \(p\)-dimensional character fixed by \(\phi\) is counted \(p\) times in the above expression. We therefore have to divide by \(p\) to finally obtain
\[
	\ch_{p, \phi} = \frac{1}{p} \left(\sum_{i = 0}^{p - 1} \size{F_{\phi, i}}\right) - \size{F_{\phi, 0} \cap \Fix(\dual{\alpha(y)})}.
\]
To end the proof, note that
\begin{align*}
	\size{F_{\phi, i}}	&= \size{\{\chi \in \dual{A} \mid \chi \circ \phi' = \chi \circ \alpha(y)^{i}\}}	\\
				&= \size{\Fix(\dual{\phi' \circ \alpha(y)^{-i}})}	\\
				&= \size{\Fix(\phi' \circ \alpha(y)^{-i})}
\end{align*}
by \cref{cor:nbFixedPointsDualMap} and that \(F_{\phi, 0} = \Fix(\dual{\phi'})\), thus
\[
	\ch_{p, \phi}(G) = \frac{1}{p} \left(\sum_{i = 0}^{p - 1} \size{\Fix(\phi' \circ \alpha(y)^{-i})}\right) - \size{\Fix(\dual{\phi'}) \cap \Fix(\dual{\alpha(y)})}.
\]
Switching from \(-i\) to \(i\) does not change the first summation, hence the result follows.
\end{proof}

The preceding results will yield necessary conditions on Reidemeister numbers of automorphisms of split metacyclic groups. The main tool to prove that these conditions are also sufficient is a number theoretic result concerning the existence of an integer satisfying certain divisor properties. The result resembles a generalisation of the \(\supseteq\)-inclusion from \cref{prop:SpecRCyclicGroups}.

\begin{theorem}	\label{theo:systemOfGCDEquations}
	Let \(l, n, m\) be non-negative integers with \(n, m \geq 1\) and \(l \leq m\) and let \(p\) be a prime not dividing \(n\). Let \(a\) be an integer such that \(\gcd(a, n) = \gcd(a - 1, n) = 1\) and such that \(\bar{a} \in \unitsb{\ZnZ}\) has order \(p\). Suppose that \(d_{0}, \ldots, d_{p - 1}\) are pairwise coprime divisors of \(n\), with, if \(p = 2\) and \(n \equiv 0 \bmod 3\), the extra condition that \(d_{0} d_{1} \equiv 0 \bmod 3\). Then there exists a \(\gamma \in \Z\) such that \(\gcd(\gamma, n) = 1\), such that \(\gcd(\gamma - a^{i}, n) = d_{i}\) for all \(i \in \{0, \ldots, p - 1\}\) and such that \(\gamma - 1 \equiv p^{l} \bmod p^{m}\).
	
\end{theorem}

\begin{proof}
Note that \(n\) must be odd, since otherwise \(\gcd(a(a - 1), n) \geq 2\). First, we consider the case where \(p\) is odd. Let \(n = \prod_{i = 1}^{r} p_{i}^{e_{i}}\) be the prime factorisation of \(n\), where all \(p_{i}\) are distinct primes. For \(0 \leq i \leq p - 1\), let \(d_{i} = \prod_{j = 1}^{r}p_{j}^{e_{i, j}}\) be the prime factorisation of \(d_{i}\). Consider, for all \(1 \leq j \leq r\) and all \(0 \leq i \leq p - 1\) such that \(e_{i, j} \geq 1\), the congruence
\begin{equation}	\label{eq:congruenceModPj}
	x - a^{i} \equiv p_{j}^{e_{i, j}} \bmod p_{j}^{e_{j}}.
\end{equation}
Since all the \(d_{i}\)'s are coprime, we have at most one congruence modulo \(p_{j}^{e_{j}}\) for each \(1 \leq j \leq r\). Let \(q\) be the product of all \(p_{j}^{e_{j}}\) not occurring in one of the congruences above, and consider (if \(q \ne 1\)), the congruence
\begin{equation}	\label{eq:congruenceModq}
	x \equiv -1 \bmod q.
\end{equation}

The system we then consider consists of all equations from \eqref{eq:congruenceModPj} and \eqref{eq:congruenceModq}, together with \(x - 1 \equiv p^{l} \bmod p^{m}\). Since each \(p_{j}^{e_{j}}\) occurs at most once and \(q\) is the product of all \(p_{j}^{e_{j}}\)'s that did not yet occur, the conditions for the Chinese Remainder Theorem are met and we find a \(\gamma \in \Z\) satisfying all these congruences. Note that the product of all moduli is equal to \(np^{m}\).

We now prove that \(\gamma\) satisfies the desired conditions.

First, let \(0 \leq i \leq p - 1\). We prove that \(\gcd(\gamma - a^{i}, n) = d_{i}\). Let \(j \in \{1, \ldots, r\}\). If \(e_{i, j} \geq 1\), it immediately follows from \eqref{eq:congruenceModPj} that \(\gcd(\gamma - a^{i}, p_{j}^{e_{j}})\) equals \(p_{j}^{e_{i, j}}\). So suppose that \(e_{i, j} = 0\) and that \(p_{j}\) divides \(\gamma - a^{i}\). If \(e_{k, j} \geq 1\) for some \(0 \leq k \leq p - 1\) (which is then necessarily distinct from \(i\)), then
\[
	\gamma - a^{i} \equiv \gamma - a^{k} \bmod p_{j}.
\]
This implies that \(a^{i - k} \equiv 1 \bmod p_{j}\). Since \(i \not \equiv k \bmod p\), also \(a \equiv 1 \bmod p_{j}\), which contradicts \(\gcd(a - 1, n) = 1\). On the other hand, if \(e_{k, j} = 0\) for all \(0 \leq k \leq p - 1\), then \(p_{j}\) divides \(q\) and we find
\[
	0 \equiv \gamma - a^{i} \equiv -1 - a^{i} \bmod p_{j}.
\]
As \(p\) and \(p_{j}\) are both odd, it is impossible for \(a^{i}\) to satisfy \(a^{i} \equiv -1 \bmod p_{j}\). Therefore, we have arrived yet again at a contradiction. We therefore conclude that \(\gcd(\gamma - a^{i}, n) = d_{i}\). 

Next, we prove that \(\gcd(\gamma, n) = 1\). Suppose that there is a \(j \in \{1, \ldots, r\}\) such that \(p_{j}\) divides \(\gamma\). If \(p_{j}\) divides \(q\), then we find using \eqref{eq:congruenceModq} that
\[
	0 \equiv \gamma \equiv - 1 \bmod p_{j},
\]
a contradiction. Thus, there is some \(0 \leq i \leq p - 1\) such that \(e_{i, j} \geq 1\). Then using \eqref{eq:congruenceModPj} we find that
\[
	0 \equiv \gamma - a^{i} \equiv - a^{i} \bmod p_{j},
\]
which is again a contradiction, as \(\gcd(a^{i}, n) = 1\). We conclude that \(\gcd(\gamma, n) = 1\).

Finally, we have that \(\gamma - 1 \equiv p^{l} \bmod p^{m}\) by construction. This finishes the case where \(p\) is odd.\newline

For the case where \(p = 2\), we consider almost the exact same system of congruences, except we replace \eqref{eq:congruenceModq} with
\begin{equation}	\label{eq:congruenceModqCasep2}
	x \equiv 3 \bmod q.
\end{equation}

	The proofs for the conditions on \(\gamma\) are almost identical as for the case \(p\) odd, so we omit them here.\qedhere

\end{proof}

\section{Split metacyclic groups of the form \(C_{n} \rtimes C_{p}\)}
We now discuss some structural properties of the groups we will investigate.
\begin{defin}
	Let \(G\) be a group. We call \(G\) \emph{split metacyclic} if there are cyclic groups \(C_{n}\) and \(C_{m}\) such that \(G \cong C_{n} \rtimes C_{m}\).
\end{defin}
We will consider split metacyclic groups of the form \((C_{n} \times C_{p^{m}}) \rtimes C_{p}\) where \(p\) is prime, 
\(m\) and \(n\) are non-negative integers and \(n\) is coprime with \(p\). Due to this last condition, \(C_{n} \times C_{p^{m}} \cong C_{np^{m}}\) is indeed cyclic.

The following is essentially a special case of \cite[Lemma~3.2]{GolasinskiGoncalves09}, but we include a full proof here due to the fact that our notation differs quite from theirs.

\begin{prop}	\label{prop:decompositionAsDirectProductSplitMetacyclic}
	Let \(n, m\) and \(p\) be natural numbers with \(n \geq 2\), \(m \geq 0\), \(p\) prime and \(\gcd(n, p) = 1\). Let \(G := (C_{n} \times C_{p^{m}})\rtimes_{\alpha} C_{p}\) be a semi-direct product with morphism \(\alpha: C_{p} \to \Aut(C_{n} \times C_{p^{m}}) \cong \unitsb{\ZnZ} \times \unitsb{\ZmodZ{p^{m}}}\).
	Put
	\[
		H := \{g \in C_{n} \mid \forall x \in C_{p}: \alpha(x)(g) = g\}.
	\]
	and let \(n = \prod_{i = 1}^{r} p_{i}^{e_{i}}\) be the prime factorisation of \(n\), with \(p_{i} \ne p_{j}\) if \(i \ne j\). Then 
	\begin{enumerate}[1)]
		\item \(H \lhd G\);
		\item there is a subset \(I \subseteq \{1, \ldots, r\}\) such that
			\[
				H = \prod_{i \in I} C_{p_{i}^{e_{i}}};
			\]
		\item \(G \cong H \times ((N \times C_{p^{m}}) \rtimes_{\tilde{\alpha}} C_{p})\) where \(\tilde{\alpha}: C_{p} \to \Aut(N)\) is the restriction of \(\alpha\) to \(\Aut(N)\) and
		\[
			N := \prod_{i \notin I} C_{p_{i}^{e_{i}}}.
		\]
	\end{enumerate}
\end{prop}
\begin{proof}
	If \(\alpha\) is the trivial map, then the result is clear. Therefore, suppose that \(\alpha\) is not the trivial map. It is clear that \(H \leq C_{n}\). Let \(x\), \(y\) and \(z\) be generators of \(C_{n}\), \(C_{p}\) and \(C_{p^{m}}\), respectively. Let \(h \in H\). Clearly, \(h^{x} = h\) and \(h^{z} = h\), and as \(h^{y} = \alpha(y)(h) = h\), we conclude that \(H \leq Z(G)\). Hence, \(H \lhd G\).
	
	Since \(\alpha\) is non-trivial, \(\alpha(y)\) has order \(p\). Also, by \cref{prop:automorphismGroupDirectProduct} the map \(\alpha\) splits as a map \(\alpha = (\alpha_{1}, \ldots, \alpha_{k}, \alpha')\) where \(\alpha_{i}: C_{p} \to \Aut(C_{p_{i}^{e_{i}}})\) and \(\alpha': C_{p} \to \Aut(C_{p^{m}})\). Fix \(i \in \{1, \ldots, r\}\). If \(x_{i}\) is a generator of \(C_{p_{i}^{e_{i}}}\), then \(\alpha_{i}(y)(x_{i}) = x_{i}^{f_{i}}\) for some \(f_{i} \in \Z\) satisfying \(f_{i}^{p} \equiv 1 \bmod p_{i}^{e_{i}}\). Suppose that \(f_{i} = \beta p_{i}^{l} + 1\) for some \(\beta, l \in \Z\) with \(\beta = 0\) or \(\gcd(p_{i}, \beta) = 1\), and \(0 \leq l < e_{i}\). Then
	\begin{equation*}
		1 \equiv (\beta p_{i}^{l} + 1)^{p} \bmod p_{i}^{e_{i}}.
	\end{equation*}
	Viewing this congruence modulo \(p_{i}^{l + 1}\), we find that \(1 \equiv p\beta p_{i}^{l} + 1 \bmod p_{i}^{l + 1}\). If \(\beta \ne 0\), then we reach a contradiction, since then \(\gcd(\beta p, p_{i}) = 1\). 
	
	From this, it follows that either \(f_{i} \equiv 1 \bmod p_{i}^{e_{i}}\) or \(f_{i} \not\equiv 1 \bmod p_{i}\), meaning that either \(\alpha_{i}(y)\) is the identity map or that \(\alpha_{i}(y)\) has no fixed points, by \cref{lem:FixedPointsCyclicGroup}. Consequently, putting \(I = \{i \in \{1, \ldots, r\} \mid f_{i} \equiv 1 \bmod p_{i}^{e_{i}}\}\), we find that
	\[
		H = \prod_{i \in I} C_{p_{i}^{e_{i}}}.
	\]
	
	Finally, put \(N = \prod_{i \not\in I} C_{p_{i}^{e_{i}}}\). By \cref{prop:automorphismGroupDirectProduct}, \(\alpha\) restricts to a map \(\tilde{\alpha}: C_{p} \to \Aut(N)\). By construction, \(H \cap (N \rtimes_{\tilde{\alpha}} C_{p}) = 1\). Thus, as \(G\) is generated by \(H\) and \((N \times C_{p^{m}}) \rtimes_{\tilde{\alpha}} C_{p}\) and as \(H\) is central, \(G\) is the internal direct product of \(H\) and \((N \times C_{p^{m}}) \rtimes_{\tilde{\alpha}} C_{p}\).
\end{proof}
\begin{cor}
	With the notations as before, we have
	\[
		\SpecR(G) = \SpecR(H) \cdot \SpecR((N \times C_{p^{m}}) \rtimes_{\tilde{\alpha}} C_{p}).
	\]
\end{cor}
\begin{proof}
	This follows immediately from \cref{cor:SpecDirectProductCoprime}.
\end{proof}
We can use \cref{prop:SpecRCyclicGroups} to determine \(\SpecR(H)\) with \(H\) as defined above. As a consequence of this proposition, we can restrict our attention to split metacyclic groups \((C_{n} \times C_{p^{m}}) \rtimes_{\alpha} C_{p}\) for which the subgroup \(H\) is trivial. This means that \(C_{p}\) acts freely (by automorphisms) on \(C_{n}\). Moreover, this also implies that each prime factor of \(n\) is strictly greater than \(p\). Indeed, in the notation above, \(\alpha_{i}(y)\) is an element of order \(p\) in \(\Aut(C_{p_{i}^{e_{i}}})\). As \(\Aut(C_{p_{i}^{e_{i}}}) \cong \unitsb{\ZmodZ{p_{i}^{e_{i}}}}\), it must hold that \(p\) divides \(p_{i}^{e_{i} - 1}(p_{i} - 1)\). Since \(p\) does not divide \(n\), \(p\) must divide \(p_{i} - 1\), showing that \(p < p_{i}\).

Next, we discuss some of the possible actions of \(C_{p}\).

\begin{lemma}	\label{lem:fixedPointFreeActionsC2OnCyclicGroup}
	If \(C_{2}\) acts freely by automorphisms on a cyclic group \(C_{n}\) with \(n\) odd, then \(C_{2}\) acts by inversion.
\end{lemma}
\begin{proof}
	Let \(y\) be the non-identity element in \(C_{2}\). Fix a generator \(x\) of \(C_{n}\) and write \(y \cdot x = x^{\gamma}\) for some \(\gamma \in \Z\). We prove that \(\gamma \equiv -1 \bmod n\). Let \(p\) be an (odd) prime factor of \(n\) and \(e\) its exponent in the prime factorisation of \(n\). Since \(\unitsb{\ZmodZ{p^{e}}}\) is cyclic, it has only one element of order \(2\), namely \(-1\). Therefore, \(\gamma \equiv \pm 1 \bmod p^{e}\). If \(\gamma \equiv 1 \bmod p^{e}\), then
	\[
		\left(x^{\frac{n}{p^{e}}}\right)^{\gamma} = x^{\frac{n \gamma}{p^{e}}} = x^{\frac{n}{p^{e}}},
	\]
	showing that \(y \cdot x^{n / p^{e}} = x^{n / p^{e}}\). As \(C_{2}\) acts without fixed points, this is a contradiction, hence \(\gamma \equiv -1 \bmod p^{e}\). As this holds for all prime factors of \(n\), we conclude that \(\gamma \equiv - 1 \bmod n\).
\end{proof}
The situation of the action of \(C_{p}\) on \(C_{p^{m}}\) is slightly different. Since we are interested in non-abelian groups, we may assume that \(m \geq 2\), as all groups of order \(p^{2}\) are abelian.
\begin{lemma}[{See \eg \cite[Theorem~4.4]{Gorenstein80}}]	\label{lem:ActionCpCp^m}
	Let \(p\) be a prime, \(m \geq 2\) an integer and let \(\alpha: C_{p} \to \Aut(C_{p^{m}})\) be an action of \(C_{p}\) on \(C_{p^{m}}\) by automorphisms. For (fixed) generators \(x, y\) of \(C_{p^{m}}\) and \(C_{p}\), put \(\alpha(y)(x) = x^{\gamma}\), where \(\gamma \in \Z\).
	\begin{enumerate}[1)]
		\item If \(m \geq 2\) and \(p \geq 3\), then \(\gamma \equiv 1 \bmod p^{m - 1}\).
		\item If \(m \geq 2\) and \(p = 2\), then \(\gamma \equiv \pm 1 \bmod 2^{m - 1}\)
	\end{enumerate}
\end{lemma}

We now introduce and fix some notation for the rest of the article. The letters \(n, m, p\) will be integers with \(n \geq 1\), \(m \geq 0\), \(p\) prime and \(\gcd(n, p) = 1\). We use \(SMC(n, m, p)\) to refer to any semi-direct product of the form \((C_{n} \times C_{p^{m}}) \rtimes_{\alpha} C_{p}\) where \(C_{p}\) acts freely on \(C_{n}\) (cf.\ observation after \cref{prop:decompositionAsDirectProductSplitMetacyclic}). Here, \(\alpha: C_{p} \to \Aut(C_{n} \times C_{p^{m}})\). We put \(N = np^{m}\) and write \(C_{N}\) for the cyclic group \(C_{n} \times C_{p^{m}}\). The symbols \(x\) and \(y\) denote fixed generators of \(C_{n} \times C_{p^{m}}\) and \(C_{p}\), respectively. Since
\[
	\Aut(C_{n} \times C_{p^{m}}) \cong \unitsb{\ZnZ} \times \unitsb{\ZmodZ{p^{m}}},
\]
we will regard \(\alpha(y)\) often as a number modulo \(N\). In particular, \(\alpha(y)\) is invertible modulo \(N\) and as \(\alpha(y)\) restricted to \(C_{n}\) has no fixed points, \cref{lem:FixedPointsCyclicGroup} implies that \(\gcd(\alpha(y) - 1, n) = 1\). As \(\alpha(y)\) has order \(p\), this moreover implies that \(\gcd(\alpha(y)^{i} - 1, n) = 1\) for all \(i \in \Z\) coprime with \(p\). We will write the action of the semi-direct product as \(y \cdot x = x^{y} = x^{\alpha(y)}\).

\begin{remark}
	In the notation of \cite{GolasinskiGoncalves09}, the group \(SMC(n, m, p)\) corresponds to \(D(np^{m}, p; \alpha(y))\).
\end{remark}

\begin{lemma}	\label{lem:C_nandC_NcharacteristicinG}
	Let \(G\) be an \(SMC(n, m, p)\). If \(n \geq 2\), then the subgroups \(C_{N}\) and \(C_{n}\) are characteristic in \(G\).
\end{lemma}
\begin{proof}
Since \(C_{N}\) is cyclic, its subgroup \(C_{n}\) is characteristic in \(C_{N}\). Since \(C_{N}\) is normal in \(G\), we infer that \(C_{n}\) is normal in \(G\) as well. Let \(g \in G\) be an element such that \(g^{n} = 1\). Let \(\pi: G \to G / C_{n}\) be the natural projection. The quotient group \(G / C_{n}\) has order coprime with \(n\), so \(\pi(g)^{n} = 1\) implies \(\pi(g) = 1\). Therefore, \(g \in C_{n}\). This implies that \(C_{n} = \{g \in G \mid g^{n} = 1\}\), implying in turn that \(C_{n}\) is characteristic in \(G\).

It is clear that \(C_{N}\) is contained in the centraliser of \(C_{n}\). Conversely, let \(x^{i} y^{j}\) be an element with \(j \not \equiv 0 \bmod p\). Note that \(C_{n}\) is generated by \(x^{p^{m}}\), and \(x^{p^{m}} \ne 1\) as \(n \geq 2\). Then
\[
	\left(x^{p^{m}}\right)^{x^{i} y^{j}} = \left(x^{p^{m}}\right)^{y^{j}} = \left(x^{p^{m}}\right)^{\alpha(y)^{j}} \ne x^{p^{m}},
\]
since \(\alpha(y)\) restricted to \(C_{n}\) has no fixed points and \(j \not \equiv 0 \bmod p\). Therefore, the centraliser of \(C_{n}\) in \(G\) equals \(C_{N}\). This implies that \(C_{N}\) is characteristic as well, since \(C_{n}\) is characteristic.
\end{proof}
\begin{lemma}	\label{lem:imageyUnderAutomorphism}
	Let \(G\) be an \(SMC(n, m, p)\) with \(n \geq 2\) and let \(\phi \in \Aut(G)\). Then there is an \(a \in \Z\) such that \(\phi(y) = x^{a}y\).
\end{lemma}
\begin{proof}
	Since \(C_{N}\) is characteristic by \cref{lem:C_nandC_NcharacteristicinG}, \(\phi(x) = x^{\gamma}\) for some \(\gamma \in \Z\) coprime with \(N\), as \(\phi\) is an automorphism. Write \(\phi(y) = x^{a}y^{b}\). Since \(\phi\) preserves the identity \(x^{y} = x^{\alpha(y)}\), we find that
	\[
		x^{\gamma \alpha(y)} = (x^{\gamma})^{x^{a}y^{b}} = y^{-b} x^{\gamma} y^{b} = (y^{-b} x y^{b})^{\gamma} = x^{\alpha(y)^{b} \gamma}.
	\]
	Thus, \(\alpha(y)^{b} \gamma \equiv \gamma \alpha(y) \bmod N\). Since \(\gamma\) and \(\alpha(y)\) are both invertible modulo \(N\), we see that \(\alpha(y)^{b-1} \equiv 1 \bmod N\). The order of \(\alpha(y)\) is \(p\), hence \(b \equiv 1 \bmod p\). Since \(y\) has order \(p\) as well, we may assume that \(b = 1\), proving the lemma.
\end{proof}

Finally, as we need the number of one-dimensional characters on \(G\), we will have to determine its commutator group.
\begin{prop}[{See \eg \cite[Corollary~3.1]{GolasinskiGoncalves09}}]	\label{prop:commutatorSubgroupSMC}
	Let \(G\) be an \(SMC(n, m, p)\). Then \([G, G] = \grpgen{x^{\alpha(y) - 1}}\).
\end{prop}

\section{Determining \(\SpecR(SMC(n, m, p))\)}
In this section, we will determine the Reidemeister spectrum of an arbitrary \(SMC(n, m, p)\). We first investigate how we can further simplify \cref{prop:generalExpressionchpphi} for \(SMC(n, m, p)\). So, let \(G\) be an arbitrary \(SMC(n, m, p)\). Then, in the notation of \cref{subsec:CharactersSemidirectProduct}, \(A = C_{N}\). Suppose that \(C_{N}\) is characteristic (for instance, when \(n \geq 2\)). Let \(\phi \in \Aut(G)\) and put \(\phi(x) = x^{\gamma}\) with \(\gcd(\gamma, N) = 1\). We have to determine \(\size{\Fix(\phi' \circ \alpha(y)^{i})}\) for \(i \in \{0, \ldots, p - 1\}\). By \cref{lem:FixedPointsCyclicGroup}, we know that
\[
	\size{\Fix(\phi' \circ \alpha(y)^{i})} = \gcd(\gamma \alpha(y)^{i} - 1, N).
\]
Strictly speaking, \(\gamma\alpha(y)^{i}\) is an integer modulo \(N\) and no true integer, but we interpret \(\gamma \alpha(y)^{i}\) as a representative of the congruence class. Since we are interested in the greatest common divisor with \(N\), this yields no problems. We can get rid of the product by multiplying by \(\alpha(y)^{-i}\), the \(i\)th power of (a representative of) the multiplicative inverse modulo \(N\) of \(\alpha(y)\). Then the gcd does not change:
\[
	\size{\Fix(\phi' \circ \alpha(y)^{i})} = \gcd(\gamma - \alpha(y)^{-i}, N)
\]

For \(\size{\Fix(\dual{\phi'}) \cap \Fix(\dual{\alpha(y)})}\), note that since \(\dual{C_{N}} \cong C_{N}\), we can fix a generator \(\chi_{1}\) of \(\dual{C_{N}}\). Every character of \(C_{N}\) is then of the form \(\chi_{a} := \chi_{1}^{a}\) for some \(a \in \Z\). Then we find for \(a, k \in \{0, \ldots, N - 1\}\) that
\[
	(\chi_{a} \circ \phi')(x^{k}) = \chi_{a}(x^{\gamma k}) = \chi_{a}(x^{k})^{\gamma} = \chi_{a\gamma}(x^{k}),
\]
hence \(\dual{\phi'}(\chi_{a}) = \chi_{a}^{\gamma}\) and similarly \(\dual{\alpha(y)}(\chi_{a}) = \chi_{a}^{\alpha(y)}\). Therefore, \cref{lem:FixedPointsCyclicGroup} yields
\[
	\size{\Fix(\dual{\phi'})} = \gcd(\gamma - 1, N), \quad \size{\Fix(\dual{\alpha(y)})} = \gcd(\alpha(y) - 1, N).
\]
\cref{lem:intersectionSubgroupsCyclicGroup} then implies that
\[
	\size{\Fix(\dual{\phi'}) \cap \Fix(\dual{\alpha(y)})} = \gcd(\gcd(\gamma - 1, N), \gcd(\alpha(y) - 1, N))
\]
which further simplifies to \(\gcd(\gamma - 1, \alpha(y) - 1, N)\). Finally, by the definition of an \(SMC(n, m, p)\), we know that \(\gcd(\alpha(y) - 1, n) = 1\), implying that
\[
	\gcd(\gamma - 1, \alpha(y) - 1, N) = \gcd(\gamma - 1, \alpha(y) - 1, p^{m}).
\]

We therefore conclude that
\begin{equation*}
	\ch_{p, \phi} = \frac{1}{p} \left(\sum_{i = 0}^{p - 1} \gcd(\gamma - \alpha(y)^{-i}, N)\right) - \gcd(\gamma - 1, \alpha(y) - 1, p^{m}).
\end{equation*}
Again, switching from \(-i\) to \(i\) does not change the summation, so we will use
\begin{equation}	\label{eq:specificExpressionchpphiSMC}
	\ch_{p, \phi} = \frac{1}{p} \left(\sum_{i = 0}^{p - 1} \gcd(\gamma - \alpha(y)^{i}, N)\right) - \gcd(\gamma - 1, \alpha(y) - 1, p^{m}).
\end{equation}

We will distinguish several cases to effectively determine the Reidemeister spectrum of all \(SMC(n, m, p)\). Each case will be treated in a similar way: we start by determining the possible values for \(\ch_{1, \phi}\). Next, we simplify \eqref{eq:specificExpressionchpphiSMC} based on the specific action (except for the final case, there we will use a different approach). Finally, we combine both results to find candidate-Reidemeister numbers and finish by deciding which ones actually occur. 

\subsection{\(\SpecR(SMC(n, m, p))\) where \(C_{p}\) acts trivially on \(C_{p^{m}}\)}
Let \(G\) be an \(SMC(n, m, p)\) with \(n \geq 2\) where \(C_{p}\) acts trivially on \(C_{p^{m}}\). Recall that each prime factor of \(n\) is strictly greater than \(p\). In particular, \(n\) is odd: \(p\) is at least \(2\), hence each prime factor of \(n\) is at least \(3\). Also, \(C_{N} = \grpgen{x}\) is characteristic in \(G\), by \cref{lem:C_nandC_NcharacteristicinG}. Moreover, \cref{prop:commutatorSubgroupSMC} states that \([G, G] = \grpgen{x^{\alpha(y) - 1}}\). We already assume that \(\gcd(\alpha(y) - 1, n) = 1\) and as \(C_{p}\) acts trivially on \(C_{p^{m}}\), \(\gcd(\alpha(y) - 1, p^{m}) = p^{m}\). Therefore, \([G, G] = \grpgen{x^{p^{m}}} = C_{n}\).
\begin{prop}	\label{prop:ExpressionFork1phiSMC(n0p)}
	Let \(\phi \in \Aut(G)\) be an automorphism. Write \(\phi(x) = x^{\gamma}\) for some \(\gamma \in \Z\) coprime with \(N\) and put \(e = \ord_{p}(\gcd(\gamma - 1, p^{m}))\). Then \(\ch_{1, \phi} \in \{p^{e}, p^{e + 1}\}\).
\end{prop}
\begin{proof}
	As \(n \geq 2\), \(\phi(y) = x^{a}y\) for some \(a \in \Z\), by \cref{lem:imageyUnderAutomorphism}. Also, \(\ab{\phi}(\bar{x}) = \bar{x}^{\gamma}\), thus \(\ab{\phi}(\bar{x},\bar{y}) = (\bar{x}^{\gamma + a}, \bar{y})\) on \(G / \gamma_{2}(G) \cong C_{p^{m}} \times C_{p}\). If \(m = 0\), then \(\ab{\phi}\) is the identity map, hence \(\ch_{1, \phi} = R(\ab{\phi}) = p = p^{e + 1}\).
	
	Suppose \(m \geq 1\). Since \(\bar{y}\) has order \(p\), so does \(\ab{\phi}(\bar{y})\), hence \(a \equiv 0 \bmod p^{m - 1}\). As we work in an abelian group, we have that \(R(\ab{\phi}) = \size{\Fix(\ab{\phi})}\), thus we need to determine the number of solutions of the congruence
	\[
		i (\gamma - 1) + ja \equiv 0 \bmod p^{m},
	\]
	where \(i \in \{0, \ldots, p^{m} - 1\}\) and \(j \in \{0, \ldots, p - 1\}\). Write \(\gamma - 1 = p^{e}\Gamma\) and \(a = p^{m - 1}A\) with \(A, \Gamma \in \Z\). Then the congruence becomes
	\[
		i \Gamma p^{e} + j A p^{m - 1} \equiv 0 \bmod p^{m}.
	\]
	First suppose that \(e = m\). Then \(i\) can take on any value, yielding \(p^{m}\) possibilities. If \(A \equiv 0 \bmod p\), then \(j\) can also take on any value, yielding \(p^{m + 1}\) total solutions. If \(A \not \equiv 0 \bmod p\), then \(j \equiv 0 \bmod p\), yielding \(1\) possibility for \(j\) and hence \(p^{m}\) solutions in total.
	
	Now suppose that \(e \leq m - 1\). Then  \(\gcd(\Gamma, p) = 1\). Modulo \(p^{m - 1}\), the congruence yields \(i \Gamma p^{e} \equiv 0 \bmod p^{m - 1}\), hence \(i \equiv 0 \bmod p^{m - e - 1}\). Write \(i = p^{m - e - 1} I\). Dividing by \(p^{m - 1}\) the congruence becomes
	\[
		I \Gamma + jA \equiv 0 \bmod p.
	\]
	Recall that \(\Gamma\) is coprime with \(p\). If \(A \equiv 0 \bmod p\), then \(j\) can take on any value and \(I \equiv 0 \bmod p\). Thus, \(i \equiv 0 \bmod p^{m - e}\), yielding \(p^{e}\) possibilities for \(i\) and a total of \(p^{e}\cdot p = p^{e + 1}\) solutions. If \(A \not \equiv 0 \bmod p\), then each value of \(j\) yields a unique value for \(I\) modulo \(p\). As \(0 \leq i \leq p^{m} - 1\), we have that \(0 \leq I \leq p^{e + 1} - 1\), hence each value of \(j\) yields \(p^{e}\) possibilities for \(i\). Thus, we find a total of \(p^{e}\cdot p = p^{e + 1}\) solutions to the original congruence.
	
	Summarised, we have
	\begin{itemize}
		\item \(p\) solutions if \(m = 0\);
		\item \(p^{e + 1}\) solutions if \(e \leq m - 1\) and \(m \geq 1\);
		\item \(p^{e + 1}\) solutions if \(e = m\) and \(a \equiv 0 \bmod p^{m}\);
		\item \(p^{e}\) solutions if \(e = m\) and \(a \not \equiv 0 \bmod p^{m}\).\qedhere
	\end{itemize}
\end{proof}
\begin{remark}
Note that \(e\) as defined above is at least \(1\) if \(p = 2\) and \(m \geq 1\), since then \(\gamma\) is odd and hence \(\gamma - 1\) is even.
\end{remark}
\begin{prop}	\label{prop:ExpressionkpphiSMC(n0p)}
	Let \(\phi \in \Aut(G)\) and write \(\phi(x) = x^{\gamma}\) for some \(\gamma \in \Z\) coprime with \(N\). Put \(e = \ord_{p}(\gcd(\gamma - 1, p^{m}))\) and put, for \(i \in \{0, \ldots, p - 1\}, d_{i} = \gcd(\gamma - \alpha(y)^{i}, N) / p^{e}\). Then \(d_{0}, \ldots, d_{p - 1}\) are pairwise coprime integers, all divide \(n\) and we have that
	\[
		\ch_{p, \phi} = p^{e - 1}\left(\sum_{i = 0}^{p - 1}d_{i}\right) - p^{e}.
	\]
	Moreover, if \(p = 2\) and \(n \equiv 0 \bmod 3\), then \(d_{0}d_{1} \equiv 0 \bmod 3\).
\end{prop}
\begin{proof}
	Fix \(\phi \in \Aut(G)\). As \(C_{N}\) is characteristic, we know by \eqref{eq:specificExpressionchpphiSMC} that
	\begin{equation}	\label{eq:expressionchpphiSMC(nmp)trivialaction}
		\ch_{p, \phi} = \frac{1}{p} \left(\sum_{i = 0}^{p - 1} \gcd(\gamma - \alpha(y)^{i}, N)\right) - \gcd(\gamma - 1, \alpha(y) - 1, p^{m}).
	\end{equation}
	Put, for \(i \in \{0, \ldots, p - 1\}\), \(a_{i} = \gcd(\gamma - \alpha(y)^{i}, N)\). Then for \(i \ne j \in \{0, \ldots, p - 1\}\) we have
	\[
		\gcd(a_{i}, a_{j}) = \gcd(\gamma - \alpha(y)^{i}, \gamma - \alpha(y)^{j}, N) = \gcd(\gamma - \alpha(y)^{i}, \alpha(y)^{i} - \alpha(y)^{j}, N).
	\]
	Without loss of generality, assume \(i > j\). Since \(\alpha(y)\) is coprime with \(N\), \(\gcd(\alpha(y)^{i} - \alpha(y)^{j}, N) = \gcd(\alpha(y)^{i - j} - 1, N)\). As \(0 < i - j < p\) and by the definition of an \(SMC(n, m, p)\), we have that \(\gcd(\alpha(y)^{i - j} - 1, n) = 1\), hence
	\[
		\gcd(\alpha(y)^{i - j} - 1, N) = \gcd(\alpha(y)^{i - j} - 1, p^{m}).
	\]
	Furthermore, we assume here that \(C_{p}\) acts trivially on \(C_{p^{m}}\), implying that \(\alpha(y) \equiv 1 \bmod p^{m}\). Thus, \(\gcd(\alpha(y)^{i - j} - 1, p^{m}) = p^{m}\), implying that
	\[
		\gcd(a_{i}, a_{j}) = \gcd(\gamma - \alpha(y)^{i}, p^{m}) = \gcd(\gamma - 1, p^{m}) = p^{e}.
	\]
	Thus, each \(d_{i} = a_{i}/p^{e}\) is an integer, and as \(a_{i}\) divides \(N\) and
	\[
		\gcd(a_{i}, p^{m}) = \gcd(\gamma - \alpha(y)^{i}, p^{m}) = \gcd(\gamma - 1, p^{m}) = p^{e}
	\]
	 for each \(i \in \{0, \ldots, p - 1\}\), the \(d_{i}\)'s are pairwise coprime divisors of \(n\).
	 	
	Similarly, we find that
	\[
		\gcd(\gamma - 1, \alpha(y) - 1, p^{m}) = \gcd(\gamma - 1, p^{m}) = p^{e}.
	\]
	We can therefore rewrite \eqref{eq:expressionchpphiSMC(nmp)trivialaction} as
	\[
		\ch_{p, \phi} = \frac{1}{p} \left(\sum_{i = 0}^{p - 1} p^{e}d_{i}\right) - p^{e} = p^{e - 1}\left(\sum_{i = 0}^{p - 1} d_{i}\right) - p^{e},
	\]

	yielding the desired expression.
	
	Finally, suppose that \(p = 2\) and that \(n \equiv 0 \bmod 3\). Since \(C_{2}\) acts without fixed points on \(C_{n}\), it acts by inversion, by \cref{lem:fixedPointFreeActionsC2OnCyclicGroup}, so \(\alpha(y) \equiv -1 \bmod n\). Then \(a_{0} = \gcd(\gamma - 1, N)\) and \(a_{1} = \gcd(\gamma + 1, N)\). As \(\gcd(\gamma, N) = 1\), it follows that \(\gamma \equiv \pm 1 \bmod 3\), hence \(a_{0} a_{1} \equiv 0 \bmod 3\), and thus also \(d_{0} d_{1} \equiv 0 \bmod 3\).
	\end{proof}
We can now fully determine the Reidemeister spectrum. For positive integers \(a\) and \(b\) we put
\begin{align*}
	D(a, b)	&= \{(d_{1}, \ldots, d_{b}) \in \Z_{\geq 0}^{b} \mid \forall i \in \{1, \ldots, b\}: d_{i} \mid a, \forall i \ne j \in \{1, \ldots, b\}: \gcd(d_{i}, d_{j}) = 1\},	\\
	D'(a, b)	&= \{(d_{1}, \ldots, d_{b}) \in D(a, b) \mid \exists i \in \{1, \ldots, b\}: d_{i} \equiv 0 \bmod 3\}.
\end{align*}
For \(d \in D(a, b)\) or \(D'(a, b)\), we define \(\Sigma(d) := \sum\limits_{i = 1}^{p} d_{i}\).
\begin{theorem}	\label{theo:SpecRSMC(n0p)}
	Let \(G\) be an \(SMC(n, m, p)\) with \(n \geq 2\) and where \(C_{p}\) acts trivially on \(C_{p^{m}}\). 
	\begin{enumerate}[1)]
		\item If \(p\) is odd and \(m = 0\), then
		\[
			\SpecR(G) = \left\{p + \frac{\Sigma(d)}{p} - 1 \middlebar d \in D(n, p)\right\}.
		\]
		\item If \(p\) is odd and \(m \geq 1\), then
		\[
		\begin{split}
			\SpecR(G) = \left\{p^{e + 1} + p^{e - 1} (\Sigma(d) - p) \middlebar 0 \leq e \leq m, d \in D(n, p)\right\} \cup{}\\
						\left\{p^{m} + p^{m - 1} (\Sigma(d) - p) \middlebar d \in D(n, p)\right\}.
		\end{split}
		\]	
		\item If \(p = 2\), \(m = 0\) and \(n \not \equiv 0 \bmod 3\), then
		\[
			\SpecR(G) = \left\{2 + \frac{\Sigma(d)}{2} - 1\middlebar d \in D(n, 2)\right\}.
		\]
		\item If \(p = 2\), \(m \geq 1\) and \(n \not \equiv 0 \bmod 3\), then
		\[
		\begin{split}
			\SpecR(G) = \left\{2^{e + 1} + 2^{e - 1} (\Sigma(d) - 2) \middlebar 1 \leq e \leq m, d \in D(n, 2)\right\} \cup{}\\
						\left\{2^{m} + 2^{m - 1} (\Sigma(d) - 2) \middlebar d \in D(n, 2)\right\}.
		\end{split}
		\]
		\item If \(p = 2\) and \(n \equiv 0 \bmod 3\), then \(D(n, p)\) has to be replaced with \(D'(n, p)\) in both expressions.
	\end{enumerate}
\end{theorem}
\begin{proof}
	For the \(\subseteq\)-inclusion for all cases of \(p\) and \(m\), we combine \cref{prop:ExpressionFork1phiSMC(n0p)} (together with the remark following it) with \cref{prop:ExpressionkpphiSMC(n0p)}.
	
	Next, we investigate the other inclusion for both cases of \(p\) and \(m\). To prove that the candidate-Reidemeister numbers do indeed lie in \(\SpecR(G)\), we have to solve the following problem: given \(d := (d_{0}, \ldots, d_{p - 1})\) and \(e\) satisfying the necessary conditions, find an integer \(\gamma \in \Z\) coprime with \(N\) such that
\[
	\gcd(\gamma - \alpha(y)^{i}, N) = p^{e}d_{i} \text{ for all } i \in \{0, \ldots, p - 1\}.
\]
Indeed, \cref{prop:ExpressionFork1phiSMC(n0p),prop:ExpressionkpphiSMC(n0p)} then imply that the map \(\phi: G \to G\) given by \(\phi(x) = x^{\gamma}, \phi(y) = y\) is an automorphism of \(G\) with Reidemeister number equal to \(p^{e + 1} + p^{e - 1}(\Sigma(d) - 1)\).
For \(e = m \geq 1\), the map \(\phi: G \to G\) given by \(\phi(x) = x^{\gamma}, \phi(y) = x^{np^{m - 1}} y\) is then an automorphism of \(G\) with Reidemeister number equal to \(p^{m} + p^{m - 1} (\Sigma(d) - p)\),
again by \cref{prop:ExpressionFork1phiSMC(n0p),prop:ExpressionkpphiSMC(n0p)}. Note that \(\phi(y)\) has indeed order \(p\), as \(x^{np^{m - 1}} \in C_{p^{m}}\) commutes with \(y\) and has order \(p\) as well.

Using \cref{theo:systemOfGCDEquations}, we can find a \(\gamma \in \Z\) such that \(\gcd(\gamma, n) = 1\) and such that \(\gcd(\gamma - \alpha(y)^{i}, n) = d_{i}\) for all \(i \in \{0, \ldots, p - 1\}\). If \(m = 0\), \(\gamma\) is the desired number and we are done. So, assume that \(m \geq 1\). Then we can also obtain from \cref{theo:systemOfGCDEquations} that \(\gamma \equiv 1 + p^{e} \bmod p^{m}\). If \(p\) were to divide \(\gamma\), then \(\gamma \equiv p^{e} + 1 \bmod p^{m}\) leads to \(0 \equiv 1 \bmod p\) or \(0 \equiv 2 \bmod p\) when \(p\) is odd, and to \(0 \equiv 1 \bmod p\) when \(p = 2\), since then \(e \geq 1\). In either case, we have a contradiction, proving that \(\gcd(\gamma, p) = 1\). It follows that \(\gcd(\gamma, N) = 1\).

Next, let \(i \in \{0, \ldots, p - 1\}\). Note that
\[
	\gamma - \alpha(y)^{i} \equiv \gamma - 1 \equiv p^{e} \bmod p^{m}
\]
since \(\alpha(y) \equiv 1 \bmod p^{m}\) as \(C_{p}\) acts trivially on \(C_{p^{m}}\). Therefore, \(\gcd(\gamma - \alpha(y)^{i}, p^{m}) = p^{e}\). Consequently, \(\gcd(\gamma - \alpha(y)^{i}, N) = p^{e}d_{i}\), which proves that \(\gamma\) is the desired number.
\end{proof}
\subsection{\(\SpecR(SMC(n, m, p))\) where \(C_{p}\) acts non-trivially on \(C_{p^{m}}\)}
Let \(G\) be an \(SMC(n, m, p)\) where \(C_{p}\) acts freely on \(C_{n}\) and non-trivially on \(C_{p^{m}}\). This implies that \(m \geq 2\). If \(p\) is odd, then \(\alpha(y) \equiv \beta p^{m - 1} + 1 \bmod p^{m}\) for some \(\beta \in \{1, \ldots, p - 1\}\), by \cref{lem:ActionCpCp^m}. If \(p = 2\), there are three possibilities:
\begin{itemize}
	\item \(\alpha(y) \equiv -1 \bmod 2^{m}\)
	\item \(\alpha(y) \equiv 2^{m - 1} - 1 \bmod 2^{m}\)
	\item \(\alpha(y) \equiv 2^{m - 1} + 1 \bmod 2^{m}\)
\end{itemize}
\begin{remark}
If \(m = 2\), the second case is the trivial action and the first and third coincide. Therefore, for \(m = 2\) we will only consider the first case.
\end{remark}

Either way, \(n\) is odd and each prime factor of \(n\) is strictly greater than \(p\). A priori, it is possible that \(n = 1\). However, we will impose further restrictions on \(n\), depending on the values of \(m\) and \(p\).

\subsubsection{\(\alpha(y) \equiv 1 \bmod p^{m - 1}\) and \(n \geq 2\)}
In this section, we assume that \(\alpha(y) \equiv \beta p^{m - 1} + 1 \bmod p^{m}\) for some \(\beta \in \{1, \ldots, p - 1\}\) and that \(n \geq 2\). Moreover, if \(p = 2\), we may assume that \(m \geq 3\), as remarked before. Since \(n \geq 2\), \cref{lem:C_nandC_NcharacteristicinG} implies that \(C_{N} = \grpgen{x}\) is characteristic in \(G\). Also, \cref{prop:commutatorSubgroupSMC} implies that \([G, G] = \grpgen{x^{p^{m - 1}}}\).
%
%

\begin{prop}	\label{prop:ExpressionFork1phiSMC(nmp)}
	For \(\phi \in \Aut(G)\), write \(\phi(x) = x^{\gamma}\) for some \(\gamma \in \Z\) coprime with \(N\). Put \(e = \ord_{p}(\gcd(\gamma - 1, p^{m}))\). Then \(\ch_{1, \phi} = p^{\min\{m, e + 1\}}\).
\end{prop}
\begin{proof}
	By \cref{lem:imageyUnderAutomorphism}, \(\phi(y) = x^{a}y\) for some \(a \in \Z\). By \cref{lem:C_nandC_NcharacteristicinG}, \(\phi\) induces an automorphism on \(\frac{G}{C_{n}} \cong C_{p^{m}} \rtimes_{\bar{\alpha}} C_{p}\), with respective generators \(\bar{x}\) and \(\bar{y}\). Since \(\bar{y}\) has order \(p\), so must \(\bar{x}^{a}\bar{y}\). Therefore,
	\begin{equation}	\label{eq:EquationFora}
		\bar{1} = (\bar{x}^{a} \bar{y})^{p} = \bar{x}^{a(1 + \alpha(y)^{-1} + \ldots + \alpha(y)^{-p + 1})}.
	\end{equation}
	Here, the inverses are to be seen modulo \(N\). Now,
	\begin{align*}
		\sum_{i = 0}^{p - 1} \alpha(y)^{-i}	&\equiv \sum_{i = 0}^{p - 1} \alpha(y)^{i}	\\
									&\equiv \sum_{i = 0}^{p - 1} (\beta p^{m - 1} + 1)^{i}	\\
									&\equiv \sum_{i = 0}^{p - 1} (\beta i p^{m - 1} + 1)	\\
									&\equiv p + \beta p^{m - 1} \frac{p(p - 1)}{2}	\bmod p^{m}
	\end{align*}
	If \(p\) is odd, \(\frac{p - 1}{2} \in \Z\), therefore, the sum reduces to \(p \bmod p^{m}\). From \eqref{eq:EquationFora} it then follows that \(ap \equiv 0 \bmod p^{m}\). Therefore, \(a \equiv 0 \bmod p^{m - 1}\).  If \(p = 2\), the sum reduces to \(2 + 2^{m - 1} \bmod 2^{m}\). As we may assume for \(p = 2\) that \(m \geq 3\), we see that \(a(2 + 2^{m - 1}) \equiv 0 \bmod 2^{m}\) implies that \(a \equiv 0 \bmod 2^{m - 1}\).
	
	Now, \(\frac{G}{\gamma_{2}(G)}\) is isomorphic to \(C_{p^{m - 1}} \times C_{p}\) since \(\gamma_{2}(G) = \grpgen{x^{p^{m - 1}}}\). Denoting the projections of \(x\) and \(y\) on the abelianisation with \(\bar{x}\) and \(\bar{y}\), the induced automorphism \(\ab{\phi}\) then satisfies
	\[
		\ab{\phi}(\bar{x}) = \bar{x}^{\gamma}, \ab{\phi}(\bar{y}) = \bar{y},
	\]
	as \(\bar{x}^{a} = \bar{1}\) in \(\frac{G}{\gamma_{2}(G)}\). By \cref{prop:ReidemeisterNumberDirectProductAutomorphismGroups,lem:FixedPointsCyclicGroup}, the Reidemeister number of \(\ab{\phi}\) is given by \(\gcd(\gamma - 1, p^{m - 1}) \cdot p = p^{\min\{e, m - 1\} + 1}\). Note that \(\gamma\) is only defined modulo \(N\), hence in particular modulo \(p^{m}\), but this does not affect \(\min\{e, m - 1\}\).
\end{proof}


\begin{prop}	\label{prop:ExpressionkpphiSMC(nmp)}
	Let \(\phi \in \Aut(G)\) and write \(\phi(x) = x^{\gamma}\) for some \(\gamma \in \Z\) coprime with \(N\). Put \(e = \ord_{p}(\gcd(\gamma - 1, p^{m}))\) and put, for \(i \in \{0, \ldots, p - 1\}, d_{i} = \gcd(\gamma - \alpha(y)^{i}, N) / p^{\min\{e, m - 1\}}\). Then the \(d_{i}\)'s are pairwise coprime integers, all divide \(np\) and we have that
	\(
		\ch_{p, \phi} = p^{\min\{e - 1, m - 2\}} (\Sigma(d) - p). 
	\)
	
	Moreover, 
	\begin{enumerate}[1)]
		\item \(d_{0}\ldots d_{p - 1} \equiv 0 \bmod p\) if and only if \(e \geq m - 1\);
		\item \(d_{0} d_{1} \equiv 0 \bmod 3\) if \(p = 2\) and \(n \equiv 0 \bmod 3\).
	\end{enumerate}	
\end{prop}
\begin{proof}
		The proof of the expression for \(\ch_{p, \phi}\) is almost identical to the one in \cref{prop:ExpressionFork1phiSMC(n0p)}, so we omit the details here. We will, however, prove the `moreover'-part. As before, let \(a_{i} = \gcd(\gamma - \alpha(y)^{i}, N)\). Then, note that
	\[
		\gcd(a_{i}, p^{m}) = \gcd(\gamma - \alpha(y)^{i}, p^{m}) = \gcd(\gamma - \beta ip^{m - 1} - 1, p^{m}).
	\]
	If \(e \leq m - 2\), then this equals \(p^{e}\) for each \(i\). If \(e = m\), then
	\[
		\gcd(\gamma - \beta ip^{m - 1} - 1, p^{m}) = \begin{cases} p^{m}	& \text{if } i = 0	\\
													p^{m - 1}	& \text{if } i \ne 0.
		\end{cases}
	\]
	Finally, if \(e = m - 1\), then write \(\gamma - 1 = p^{m - 1} \Gamma\) with \(\gcd(\Gamma, p) = 1\) (\(\Gamma \ne 0\) as \(\gamma \ne 1\) in that case). Then
	\[
		\gcd(\gamma - \beta ip^{m - 1} - 1, p^{m}) = p^{m - 1}\gcd(\Gamma - \beta i, p) = \begin{cases} p^{m}	& \text{if } \Gamma - \beta i \equiv 0 \bmod p	\\
													p^{m - 1}	& \text{otherwise.}
		\end{cases}
	\]
	This proves the first `moreover'-claim. The fact that \(d_{0}d_{1} \equiv 0 \bmod 3\) if \(p = 2\) and \(n \equiv 0 \bmod 3\) is proven similarly as in \cref{prop:ExpressionkpphiSMC(n0p)}.
	\end{proof}
	
\begin{theorem}	\label{theo:SpecRSMC(nmp)}
	Let \(G\) be an \(SMC(n, m, p)\) with \(n \geq 2\) and \(\alpha(y) \equiv \beta p^{m - 1} + 1 \bmod p^{m}\) for some \(\beta \in \{1, \ldots, p - 1\}\). In particular, \(m \geq 2\).
	\begin{enumerate}[1)]
	\item If \(p\) is odd, then
	\[
	\begin{split}
		\SpecR(G) = \left\{p^{e + 1} + p^{e - 1} (\Sigma(d) - p) \middlebar 0 \leq e \leq m - 1, d \in D(np, p), \right. \\ d_{1}\ldots d_{p} \equiv 0 \bmod p \iff e = m - 1\big\}.
	\end{split}
	\]
	\item If \(p = 2\), then
	\[
	\begin{split}
		\SpecR(G) = \left\{2^{e + 1} + 2^{e - 1} (\Sigma(d) - 2) \middlebar 1 \leq e \leq m - 1, d \in D(2n, 2),\right. \\ d_{1}d_{2} \equiv 0 \bmod 2 \iff e = m - 1\big\}.
	\end{split}
	\]
	\item If \(p = 2\) and \(n \equiv 0 \bmod 3\), then \(D(2n, 2)\) is to be replaced by \(D'(2n, 2)\).
	\end{enumerate}
\end{theorem}
\begin{proof}
	The \(\subseteq\)-inclusion is proven by combining \cref{prop:ExpressionFork1phiSMC(nmp)} and \cref{prop:ExpressionkpphiSMC(nmp)}, with the following remarks
	\begin{itemize}
		\item It follows from the expressions for \(\ch_{1, \phi}\) and \(\ch_{p, \phi}\) that we can assume that \(e \leq m - 1\).
		\item The condition \(d_{1}\ldots d_{p} \equiv 0 \bmod p \iff e = m - 1\) is the first `moreover'-part of \cref{prop:ExpressionkpphiSMC(nmp)}.
		\item If \(\phi \in \Aut(G)\), then \(\phi(x) = x^{\gamma}\) for some \(\gamma \in \Z\) coprime with \(N\), by \cref{lem:C_nandC_NcharacteristicinG}. If \(p = 2\), this implies that \(\gamma\) is odd, therefore, \(\gamma - 1\) is even. As the exponent \(e\) comes from \(\ord_{2}(\gamma - 1)\), this shows why \(1 \leq e \leq m - 1\) for \(p = 2\).
	\end{itemize}
	
	
For the other inclusion, let \(d_{0}, \ldots, d_{p - 1}\) be divisors of \(np\) satisfying the conditions of \cref{prop:ExpressionkpphiSMC(nmp)} and let \(e\) be an integer satisfying the given inequalities. For \(i \in \{0, \ldots, p - 1\}\), put \(d_{i}' = \frac{d_{i}}{p^{\ord_{p}(d_{i})}}\). Then \cref{theo:systemOfGCDEquations} provides us with a \(\gamma \in \Z\) satisfying \(\gcd(\gamma, n) = 1\), \(\gamma \equiv 1 + p^{e} \bmod p^{m}\) and \(\gcd(\gamma - \alpha(y)^{i}, n) = d_{i}'\) for each \(i \in \{0, \ldots, p - 1\}\). Using the same argument as in \cref{theo:SpecRSMC(n0p)}, it follows that \(\gcd(\gamma, N) = 1\).

Now, for \(\gcd(\gamma - \alpha(y)^{i}, p^{m})\), first suppose that \(e \leq m - 2\). Then none of the \(d_{i}\)'s is divisible by \(p\) meaning that \(d_{i} = d_{i}'\) for all \(i \in \{0, \ldots, p - 1\}\). Let \(0 \leq i \leq p - 1\). Since \(e \leq m - 2\), the congruences \(\gamma - 1 \equiv p^{e} \bmod p^{m}\) and \(\alpha(y) \equiv 1 \bmod p^{m - 1}\) together imply that
\[
	\gamma - \alpha(y)^{i} \equiv \gamma - 1 \equiv p^{e} \bmod p^{m - 1},
\]
showing that the highest power of \(p\) dividing \(\gamma - \alpha(y)^{i}\) is \(p^{e}\). Therefore, \(\gcd(\gamma - \alpha(y)^{i}, N) = p^{e} d_{i}\) as desired.

For \(e = m - 1\), we may assume, without loss of generality, that \(d_{\inv{\beta} \bmod p} \equiv 0 \bmod p\), meaning that \(d_{j}' = d_{j}\) for \(j \not \equiv \inv{\beta} \bmod p\) and \(d_{\inv{\beta} \bmod p} = p d_{\inv{\beta} \bmod p}'\). For \(i \in \{0, \ldots, p - 1\}\), we compute \(\gamma - \alpha(y)^{i} \bmod p^{m}\):
\begin{align*}
	\gamma - \alpha(y)^{i}	&\equiv \gamma - (\beta p^{m - 1} + 1)^{i}	\\
						&\equiv \gamma - i \beta p^{m - 1} - 1	\\
						&\equiv p^{m - 1} - i \beta p^{m - 1}	\\
						& \equiv p^{m - 1}(1 - i\beta)	\bmod p^{m},
\end{align*}
where the second congruence follows from the fact that \(m \geq 2\), and the third from the fact that \(\gamma - 1 \equiv p^{m - 1} \bmod p^{m}\). We conclude that \(\gamma - \alpha(y)^{i}\) is always divisible by \(p^{m - 1}\) and that it is divisible by \(p^{m}\) if and only if the last expression is \(0\) modulo \(p^{m}\), hence if and only if \(1 - i \beta \equiv 0 \bmod p\); in other words, if and only if \(i \equiv \inv{\beta} \bmod p\). This ensures that \(\gcd(\gamma - \alpha(y)^{i}, N) = p^{e}d_{i}\).

\cref{prop:ExpressionFork1phiSMC(nmp),prop:ExpressionkpphiSMC(nmp)} then imply that the map \(\phi: G \to G\) given by \(\phi(x) = x^{\gamma}, \phi(y) = y\) is an automorphism of \(G\) with Reidemeister number equal to \(p^{e + 1} + p^{e - 1}(\Sigma(d) - p)\).
\end{proof}
\subsubsection{\(p = 2\) and \(\alpha(y) \equiv -1 \bmod 2^{m - 1}\)}
By \cref{lem:fixedPointFreeActionsC2OnCyclicGroup}, \(C_{2}\) acts on \(C_{n}\) by inversion. From the congruences of \(\alpha(y)\) modulo \(n\) and \(2^{m - 1}\), we find that \(\alpha(y) \equiv -1 \bmod 2^{m - 1}n\). This yields two possibilities for \(\alpha(y) \bmod N\), namely \(\alpha(y) \equiv -1\) or \(2^{m - 1}n - 1 \bmod N\). In either case, \cref{prop:commutatorSubgroupSMC} yields that \([G, G] = \grpgen{x^{2}}\). We do not put any restrictions on \(n\), so \(n = 1\) is possible. Also recall that \(m \geq 2\) and if \(\alpha(y) \equiv 2^{m - 1}n - 1 \bmod N\), we may assume \(m \geq 3\). This in particular implies that \(\alpha(y) \equiv -1 \bmod 4\) in either case.

The expression for \(\SpecR(SMC(n, m, 2))\) is nearly identical in both cases. Therefore, we treat them together, pointing out the differences when they occur.
\begin{lemma}	\label{lem:CNcharacteristicCase1}
	The subgroup \(C_{N}\) is characteristic in \(G\).
\end{lemma}
\begin{proof}
	Let \(a \in \Z\). Then
	\[
		(x^{a}y)^{2} = x^{a}yx^{a}y = x^{a} x^{\alpha(y)a} = x^{a(1 + \alpha(y))}.
	\]
	If \(\alpha(y) \equiv -1 \bmod 2^{m}\), the element \(x^{a}y\) has order \(2\), whereas \(x\) has order \(n \cdot 2^{m} > 2\). If \(\alpha(y) \equiv 2^{m - 1} - 1 \bmod 2^{m}\), the element \(x^{a}y\) has order either \(2\) or \(4\), whereas \(x\) has order \(n \cdot 2^{m} > 4\) (recall that \(m \geq 3\) in that case). Either way, every element of the same order as \(x\) lies in \(C_{N}\), showing that \(C_{N}\) is characteristic.
\end{proof}
\begin{prop}	\label{prop:k1phiSMC(nm2)Case1}
	Let \(\phi\) be an automorphism of \(G\). Then \(\ch_{1, \phi} \in \{2, 4\}\). Moreover, if \(\alpha(y) \equiv 2^{m - 1} - 1 \bmod 2^{m}\), then \(\ch_{1, \phi} = 4\).
\end{prop}
\begin{proof}
	By \cref{lem:CNcharacteristicCase1}, we know that \(\phi(x) = x^{\gamma}\) for some odd \(\gamma \in \Z\). Consequently, \(\ab{\phi}(\bar{x}) = \bar{x}\), as \([G, G] = \grpgen{x^{2}}\), proving that \(\ab{\phi}\) has at least \(2\) fixed points. Since \(\ab{G}\) is abelian and has size \(4\), we conclude that \(\ch_{1, \phi} = R(\ab{\phi}) \in \{2, 4\}\).
	
	Now assume \(\alpha(y) \equiv 2^{m - 1} - 1 \bmod 2^{m}\). Writing \(\phi(y) = x^{c}y^{d}\), we see that \(d = 1\), otherwise \(y \notin \im(\phi)\). The element \(\phi(y)\) must have order \(2\), thus
	\[
		1 = (x^{c}y)^{2} = x^{2^{m - 1}nc}
	\]
	showing that \(c \equiv 0 \bmod 2\). Consequently, \(\ab{\phi}(\bar{y}) = \bar{y}\). We conclude that \(\ab{\phi}\) is the identity map, therefore, \(\ch_{1, \phi} = R(\ab{\phi}) = 4\).
\end{proof}
\begin{prop}	\label{prop:k2phiSMC(nm2)Case1}
	Let \(\phi \in \Aut(G)\). Then there exist coprime positive integers \(d_{0}, d_{1}\) dividing \(n\) and an integer \(e \geq 1\) such that
	\[
		\ch_{2, \phi} = 2^{e}d_{0} + d_{1} - 2.
	\]
	Moreover, if \(n \equiv 0 \bmod 3\), then \(d_{0} d_{1} \equiv 0 \bmod 3\).
\end{prop}
\begin{proof}
	By \cref{lem:CNcharacteristicCase1}, we can write \(\phi(x) = x^{\gamma}\) for some \(\gamma\) coprime with \(N\).
	
	Again, we can use \eqref{eq:specificExpressionchpphiSMC} to determine \(\ch_{2, \phi}\), since \(C_{N}\) is characteristic. The expression reads
	\[
		\ch_{2, \phi} = \frac{\gcd(\gamma - 1, N) + \gcd(\gamma - \alpha(y), N)}{2} - \gcd(\gamma - 1, \alpha(y) - 1, 2^{m}).
	\]
	The last term equals \(2\), as \(\alpha(y) - 1 \equiv 2 \bmod 4\) and \(\gamma\) is odd. For the first two terms, remark that exactly one of the numbers \(\gamma - 1\) or \(\gamma - \alpha(y)\) is a multiple of \(4\). Indeed, both are even andtheir difference \(-\alpha(y) + 1\) is congruent to \(2\) modulo \(4\). The other is then equal to \(2a\) for some odd number \(a\). We may assume that \(\gamma - 1\) is a multiple of \(4\), by noting that \(\tau_{y} \circ \phi\) is an automorphism mapping \(x\) to \(x^{\alpha(y)\gamma}\) that has the same Reidemeister number as \(\phi\) by \cref{lem:ReidemeisterNumberCompositionInnerAutomorphism}.
	
	Thus, write \(\gcd(\gamma - 1, N) = 2^{f} d_{0}\) and \(\gcd(\gamma - \alpha(y), N) = 2d_{1}\) with \(f \geq 2\) and \(d_{0}, d_{1}\) both dividing \(n\). We get
	\begin{align*}
		\ch_{2, \phi}	&= \frac{2^{f}d_{0} + 2d_{1}}{2} - 2	\\
					&= 2^{f - 1}d_{0} + d_{1} - 2.
	\end{align*}
	Recall that \(n\) is odd, hence so are \(d_{0}\) and \(d_{1}\). Thus, \(d := \gcd(d_{0}, d_{1})\) is an odd divisor of \(n\) dividing both \(\gamma - 1\) and \(\gamma - \alpha(y)\). Hence, \(d\) divides \(\alpha(y) - 1\), which is coprime with \(n\), so \(d\) must be \(1\). Putting \(e = f - 1\) yields the desired expression for \(\ch_{2, \phi}\). The `moreover'-part is again proven similarly as in \cref{prop:ExpressionkpphiSMC(n0p)}.
\end{proof}

\begin{theorem}	\label{theo:SpecRSMC(nm2)Case1}
	Let \(G\) be an \(SMC(n, m, 2)\) where \(\alpha(y) \equiv -1 \bmod 2^{m - 1}n\) and \(m \geq 2\). 
	\begin{enumerate}[1)]
		\item If \(\alpha(y) \equiv -1 \bmod N\) and \(n \not\equiv 0 \bmod 3\), then
		\[
			\SpecR(G) = \{2, 4\} + \{2^{e}d_{0} + d_{1} - 2 \mid 1 \leq e \leq m - 1, (d_{0}, d_{1}) \in D(n, 2)\}.
		\]
		\item If \(\alpha(y) \equiv 2^{m - 1}n - 1 \bmod N\) (and thus \(m \geq 3\)) and \(n \not \equiv 0 \bmod 3\), then
		\[
			\SpecR(G) = \{2^{e}d_{0} + d_{1} + 2 \mid 1 \leq e \leq m - 1, (d_{0}, d_{1}) \in D(n, 2)\}.
		\]
		\item If \(n \equiv 0 \bmod 3\), then \(D(n, 2)\) is to be replaced by \(D'(n, 2)\) in both cases.
	\end{enumerate}
\end{theorem}
\begin{proof}
	Combining \cref{prop:k1phiSMC(nm2)Case1} and \cref{prop:k2phiSMC(nm2)Case1} yields the inclusion
	\[
		\SpecR(G) \subseteq \{2, 4\} + \{2^{e}d_{0} + d_{1} - 2 \mid 1 \leq e \leq m - 1, (d_{0}, d_{1}) \in D(n, 2)\}
	\]
	for the first case,
	\[
		\SpecR(G) \subseteq \{4\} + \{2^{e}d_{0} + d_{1} - 2 \mid 1 \leq e \leq m - 1, (d_{0}, d_{1}) \in D(n, 2)\}
	\]
	for the second case and the same with \(D'(n, 2)\) instead of \(D(n, 2)\) if \(n \equiv 0 \bmod 3\).
	
	For the other inclusion, we have to solve a similar problem as before: given \(2\) divisors \(d_{0}, d_{1}\) of \(n\) satisfying the conditions of \cref{prop:k2phiSMC(nm2)Case1} and an integer \(e\) satisfying \(1 \leq e \leq m - 1\), find an integer \(\gamma \in \Z\) coprime with \(N\) such that
\[
	\gcd(\gamma - 1, N) = 2^{e + 1}d_{0} \text{ and } \gcd(\gamma - \alpha(y), N) = 2d_{1}.
\]
Indeed, by \cref{prop:k1phiSMC(nm2)Case1,prop:k2phiSMC(nm2)Case1}, the map \(\phi: G \to G\) given by \(\phi(x) = x^{\gamma}, \phi(y) = y\) is then an automorphism of \(G\) with Reidemeister number equal to
\[
	4 + 2^{e} d_{0} + d_{1} - 2
\]
and if \(\alpha(y) \equiv -1 \bmod N\), \(\psi: G \to G\) given by \(\psi(x) = x^{\gamma}, \psi(y) = xy\) is one with Reidemeister number
\[
	2 + 2^{e} d_{0} + d_{1} - 2.
\]

\cref{theo:systemOfGCDEquations} provides us with a \(\gamma \in \Z\) satisfying \(\gcd(\gamma, n) = 1\), \(\gcd(\gamma - 1, n) = d_{0}\), \(\gcd(\gamma - \alpha(y), n) = d_{1}\) and \(\gamma \equiv 1 + 2^{e + 1} \bmod 2^{m}\). For \(\gcd(\gamma, 2)\), observe that
\[
	\gamma \equiv 2^{e + 1} + 1 \bmod 2^{m}
\]
which becomes \(\gamma \equiv 1 \bmod 2\), proving that \(\gcd(\gamma, 2) = 1\). Therefore, \(\gcd(\gamma, N) = 1\).

Clearly, \(\gcd(\gamma - 1, 2^{m}) = 2^{e + 1}\). Since \(e \geq 1\) and \(\alpha(y) \equiv 3 \bmod 4\), we find that
\[
	\gamma - \alpha(y) \equiv \gamma - 3 \equiv 2 \bmod 4.
\]
Therefore, \(\gcd(\gamma - \alpha(y), 2^{m}) = 2\). As \(\gcd(\gamma - \alpha(y)^{i}, n) = d_{i}\), we deduce that \(\gcd(\gamma - 1, N) = 2^{e + 1} d_{0}\) and \(\gcd(\gamma -\alpha(y), N) = 2 d_{1}\).
\end{proof}
\begin{remark}
	As we did not put any restrictions on \(n\), \cref{theo:SpecRSMC(nm2)Case1} also fully determines the Reidemeister spectrum of the groups \(C_{2^{m}} \rtimes_{\alpha} C_{2}\) where \(\alpha(y) \equiv -1 \bmod 2^{m - 1}\).
\end{remark}
The last remaining case is \(\alpha(y) \equiv 1 \bmod p^{m - 1}\) and \(n = 1\), which we treat in the next section.
\subsection{\(\SpecR(SMC(1, m, p))\) with \(m \geq 2\) and non-trivial action}

Here we deal with \(p\)-groups of the form \(C_{p^{m}} \rtimes_{\alpha} C_{p}\), where we assume that \(\alpha(y) \not \equiv 1 \bmod p^{m}\). In particular, \(m \geq 2\). We distinguish two cases: \(\alpha(y) \equiv 1 \bmod p^{m - 1}\) and \(\alpha(y) \not \equiv 1 \bmod p^{m - 1}\). By \cref{lem:ActionCpCp^m}, the latter can only occur if \(p = 2\). However, by the remark following the proof of \cref{theo:SpecRSMC(nm2)Case1}, the case \(\alpha(y) \not\equiv 1 \bmod 2^{m - 1}\) has already been dealt with. Therefore, we may assume that \(\alpha(y) \equiv 1 \bmod p^{m - 1}\).

We will not be able to use \cref{prop:generalExpressionchpphi} nor \eqref{eq:specificExpressionchpphiSMC}, since \(C_{p^{m}}\) is not necessarily characteristic in \(G\); for instance, this is the case for \(C_{p^{2}} \rtimes C_{p}\). We can, however, still use the technique of character counting, but we will need another approach.

By \cref{prop:commutatorSubgroupSMC}, we already know the commutator subgroup of \(G\), namely \(\grpgen{x^{p^{m - 1}}}\). By \cite[Corollary~3.1]{GolasinskiGoncalves09}, the centre of \(G\) is given by \(\grpgen{x^{p}}\).
\begin{lemma}	\label{lem:k1phiSMC(1mp)}
	Let \(\phi \in \Aut(G)\) and write \(\phi(x) = x^{a}y^{b}\), \(\phi(y) = x^{c}y^{d}\). Then
	\begin{enumerate}[1)]
		\item \(c \equiv 0 \bmod p^{m - 1}\), \(\gcd(a, p) = 1\) and \(d \equiv 1 \bmod p\);
		\item \(\ch_{1, \phi} \in \{p^{i} \mid 1 \leq i \leq m\}\) if \(p\) is odd, \(\ch_{1, \phi} \in \{2^{i} \mid 2 \leq i \leq m\}\) if \(p = 2\);
		\item \(\ch_{1, \phi} = p^{m}\) if and only if \(a \equiv 1 \bmod p^{m - 1}\) and \(b \equiv 0 \bmod p\).
	\end{enumerate}
\end{lemma}
\begin{proof}
	For the first item, we refer the reader to \cite[Proposition~1]{Schulte01}.

	For the second item, since \(\Fix(\ab{\phi})\) is a subgroup of \(\ab{G}\), it follows that \(\ch_{1, \phi} = R(\ab{\phi})\) divides \(p^{m}\). We furthermore note that \(\ab{\phi}(\bar{y}^{i}) = \bar{x}^{ci}\bar{y}^{i} = \bar{y}^{i}\) for all \(i \in \{0, \ldots, p - 1\}\), as \(c \equiv 0 \bmod p^{m - 1}\) and \(d \equiv 1 \bmod p\). Therefore, \(\ab{\phi}\) has at least \(p\) fixed points, implying that \(R(\ab{\phi}) = \size{\Fix(\ab{\phi})} \geq p\). If \(p = 2\), we also have
	\[
		\ab{\phi}(\bar{x}^{2^{m - 2}}) = \bar{x}^{2^{m - 2}}
	\]
	as \(a \equiv 1 \bmod 2\), \(y\) has order \(2\) and \(m \geq 3\) in that case. This implies that \(\ch_{1, \phi} \geq 4\) if \(p = 2\).
	The final item follows immediately from the fact that \(\ch_{1, \phi} = p^{m}\) if and only if \(\ab{\phi}\) is the identity map.
\end{proof}

Now, if \(\chi_{a}\) is a character on \(C_{p^{m}}\) inducing a \(1\)-dimensional character on \(G\), then
\[
	\chi_{a}(x^{e}) = \chi_{a}(x^{\alpha(y)^{i}e})
\]
for all \(e \in \{0, \ldots, p^{m} - 1\}\) and \(i \in \{0, \ldots, p - 1\}\). This can only hold if \(\chi_{a} = \chi_{a \alpha(y)^{i}}\) for all \(i \in \{0, \ldots, p - 1\}\), \ie if and only if \(a(\alpha(y)^{i} - 1) \equiv 0 \bmod p^{m}\). As \(\alpha(y) \equiv \beta p^{m - 1} + 1 \bmod p^{m}\) for some \(\beta \in \{1, \ldots, p - 1\}\), this is equivalent to \(a \equiv 0 \bmod p\). Therefore, the characters \(\chi_{a}\) on \(C_{p^{m}}\) inducing a \(p\)-dimensional character on \(G\) are those for which \(\gcd(a, p) = 1\).

\begin{lemma}	\label{lem:charactersSMC(1mp)}
	Let \(\overline{\chi}_{l} \in \Irr_{p}(G)\) be a character with \(\gcd(l, p) = 1\) and let \(\zeta = \exp\left(\frac{2\pi i}{p^{m}}\right)\) be a primitive \(p^{m}\)-th root of unity. Write \(\alpha(y) \equiv \beta p^{m - 1} + 1 \bmod p^{m}\) for a \(\beta \in \{1, \ldots, p - 1\}\). For \(e \in \{0, \ldots, p^{m} - 1\}\), we have
	\[
		\overline{\chi}_{l}(x^{e}) = \zeta^{le} \sum_{j = 0}^{p - 1} \zeta^{le \beta j p^{m - 1}} = \begin{cases}
			0	&	\mbox{if } \gcd(e, p) = 1	\\
			p \zeta^{le}	&	\mbox{if } \gcd(e, p) = p.
		\end{cases}
	\]
	In particular, \(\overline{\chi}_{l}\) is zero outside \(Z(G)\).
\end{lemma}
\begin{proof}
	On \(C_{p^{m}}\), the character \(\chi_{l}\) satisfies \(\chi_{l}(x^{e}) = \zeta^{le}\) for all \(e \in \{0, \ldots, p^{m} - 1\}\). Now let \(e \in \{0, \ldots, p^{m} - 1\}\). We compute \(\overline{\chi}_{l}(x^{e})\) using \cref{lem:pCharacterACp}:
	\begin{align*}
		\overline{\chi}_{l}(x^{e})	&=	\sum_{j = 0}^{p - 1} \chi_{l}(x^{e\alpha(y)^{j}})	\\
							&=	\sum_{j = 0}^{p - 1} \chi_{l}\left(x^{e(\beta j p^{m - 1} + 1)}\right)	\\
							&=	\sum_{j = 0}^{p - 1} \zeta^{le \beta j p^{m - 1} + le}	\\
							&=	\zeta^{le}\sum_{j = 0}^{p - 1} \zeta^{le \beta j p^{m - 1}}.
	\end{align*}
	Recall that \(\beta\) is coprime with \(p\). Thus, if \(\gcd(e, p) = 1\), the last sum equals \(1 + \zeta^{p^{m - 1}} + \ldots + \zeta^{(p - 1)p^{m - 1}}\), which is zero. If \(\gcd(e, p) = p\), the last sum equals \(p\), as \(\zeta^{p^{m}} = 1\). 

	Since \(Z(G) = \grpgen{x^{p}}\) and \(\overline{\chi}_{l}\) is already zero outside \(C_{p^{m}}\), we conclude that \(\overline{\chi}_{l}\) is zero outside \(Z(G)\).
\end{proof}

\begin{prop}	\label{prop:kpphiSMC(1mp)}
	Let \(\phi \in \Aut(G)\) and \(\overline{\chi}_{a}\) be a character in \(\Irr_{p}(G)\). If \(\overline{\chi}_{a}\) is fixed by \(\phi\), then all characters in \(\Irr_{p}(G)\) are fixed by \(\phi\).
	
	In particular, \(\ch_{p, \phi}(G) \in \{0, p^{m - 2}(p - 1)\}\).
\end{prop}
\begin{proof}
	Let \(\overline{\chi}_{a}\) be a character in \(\Irr_{p}(G)\). By \cref{lem:charactersSMC(1mp)}, we know that \(\overline{\chi}_{a}\) is zero outside \(Z(G)\). As \(Z(G)\) is characteristic in \(G\), also \(\phi(G \setminus Z(G)) = G \setminus Z(G)\), implying that \(\overline{\chi}_{a} \circ \phi\) is also zero outside \(Z(G)\). Since \(Z(G) = \grpgen{x^{p}}\), we can write \(\phi(x^{p}) = x^{\gamma p}\) for some \(\gamma \in \Z\) coprime with \(p\). Then, using \cref{lem:charactersSMC(1mp)} again, we find that
	\[
		\overline{\chi}_{a} \circ \phi = \overline{\chi}_{a} \iff \forall e \in \Z: p \zeta^{ape} = p \zeta^{a \gamma p e},
	\]
	where \(\zeta = \exp\left(\frac{2\pi i}{p^{m}}\right)\) is a primitive \(p^{m}\)-th root of unity. The last condition is clearly equivalent with \(\zeta^{ap(\gamma - 1)} = 1\).

	So, suppose that \(\overline{\chi}_{a} \circ \phi = \overline{\chi}_{a}\). Then \(\zeta^{ap(\gamma - 1)} = 1\). Let \(b \in \{0, \ldots, p^{m} - 1\}\) be coprime with \(p\). Then there exists a \(c \in \Z\) such that \(ac \equiv b \bmod p^{m}\). Consequently,
	\[
		\zeta^{bp(\gamma - 1)} = \zeta^{acp(\gamma - 1)} = \left(\zeta^{ap(\gamma - 1)}\right)^{c} = 1,
	\]
	proving that \(\overline{\chi}_{b} \circ \phi = \overline{\chi}_{b}\).
	
	Thus, either all or none of the characters in \(\Irr_{p}(G)\) are fixed by \(\phi\). Since
	\[
		\size{G} = \ch_{1} + p^{2}\ch_{p} = \size{\ab{G}} + p^{2}\ch_{p} = p^{m} + p^{2}\ch_{p},
	\]
	we find that \(\ch_{p} = \frac{p^{m + 1} - p^{m}}{p^{2}} = p^{m - 1} - p^{m - 2} = p^{m - 2}(p - 1)\). Thus, either \(\ch_{p, \phi} = 0\) or \(\ch_{p, \phi} = p^{m - 2}(p - 1)\).
\end{proof}

\begin{theorem}
	Let \(G\) be an \(SMC(1, m, p)\) with \(m \geq 2\) if \(p\) is odd, \(m \geq 3\) if \(p = 2\), and \(\alpha(y) \equiv \beta p^{m - 1} + 1 \bmod p^{m}\) for some \(\beta \in \{1, \ldots, p - 1\}\).
	\begin{enumerate}[1)]
		\item If \(p\) is odd, then
		\[
			\SpecR(G) = \{p^{i} \mid 1 \leq i \leq m - 1\} \cup \{2p^{m - 1} - p^{m - 2}, p^{m} + p^{m - 1} - p^{m - 2}\}.
		\]
		\item If \(p = 2\), then
		\[
			\SpecR(G) = \{2^{i} \mid 2 \leq i \leq m - 1\} \cup \{2^{m} - 2^{m - 2}, 2^{m} + 2^{m - 2}\}.
		\]
	\end{enumerate}
\end{theorem}
\begin{proof}
	Combining \cref{lem:k1phiSMC(1mp),prop:kpphiSMC(1mp)} yields the inclusions
	\[
		\SpecR(G) \subseteq \{p^{i} \mid 1 \leq i \leq m\} + \{0, p^{m - 2}(p - 1)\}
	\]
	for odd \(p\) and
	\[
		\SpecR(G) \subseteq \{2^{i} \mid 2 \leq i \leq m\} + \{0, 2^{m - 2}\}
	\]
	for \(p = 2\). For the \(\subseteq\)-inclusion, we therefore still need to prove the following:
	\begin{itemize}
		 \item If \(\ch_{p, \phi} = p^{m - 2}(p - 1)\) for \(\phi \in \Aut(G)\), then \(\ch_{1, \phi} \in \{p^{m - 1}, p^{m}\}\).
		 \item If \(\ch_{1, \phi} = p^{m}\) for \(\phi \in \Aut(G)\), then \(\ch_{p, \phi} = p^{m - 2}(p - 1)\).
	 \end{itemize}
	Fix \(\phi \in \Aut(G)\). First, suppose that \(\ch_{p, \phi} = p^{m - 2}(p - 1)\). Write \(\phi(x) = x^{a}y^{b}\) with \(a \in \{0, \ldots, p^{m} - 1\}\) and \(b \in \{0, \ldots, p - 1\}\). It is easily checked that \(\phi(x^{p}) = x^{p a}\) for odd \(p\) and \(\phi(x^{2}) = x^{(1 + \alpha(y)^{b})a}\) for \(p = 2\). As all characters in \(\Irr_{p}(G)\) are fixed, it follows from \cref{lem:charactersSMC(1mp)} that for \(p\) odd,
	\[
		p \zeta^{lpe} = p\zeta^{la pe}
	\]
	for all \(e \in \{0, \ldots, p^{m - 1} - 1\}\) and \(l\) coprime with \(p\), where again \(\zeta = \exp\left(\frac{2\pi i}{p^{m}}\right)\). This can only hold if \(\zeta^{p(a - 1)} = 1\), \ie if \(a \equiv 1 \bmod p^{m - 1}\). For \(p = 2\), we get
	\[
		2 \zeta^{2le} = 2 \zeta^{(1 + \alpha(y)^{b})ale}
	\]
	for all \(e \in \{0, \ldots, p^{m - 1} - 1\}\) and \(l\) coprime with \(p\). This can only hold if \(a + a \alpha(y)^{b} - 2 \equiv 0 \bmod 2^{m}\). Viewing this modulo \(2^{m - 1}\) yields \(2a - 2 \equiv 0 \bmod 2^{m - 1}\), showing that \(a - 1 \equiv 0 \bmod 2^{m - 2}\). Thus, for all \(p\), we have that \(a - 1 \equiv 0 \bmod p^{m - 2}\).
	
	Now, writing \(\phi(y) = x^{c} y\) with \(c \equiv 0 \bmod p^{m - 1}\) (due to \cref{lem:k1phiSMC(1mp)}), we now determine \(\ch_{1, \phi}\) by counting the number of fixed points of \(\phi^{ab}\) on \(\ab{G}\). In order to do so, we have to count the number of solutions \((A, B)\) in \(\ZmodZ{p^{m - 1}} \times \ZmodZ{p}\) to the system of congruences
	\[
		\begin{cases}
			a A \equiv A \bmod p^{m - 1}	\\
			b A + B \equiv B \bmod p.
		\end{cases}
	\]
	If \(A \equiv 0 \bmod p\), then \(A(a - 1) \equiv 0 \bmod p^{m - 1}\). Therefore, this systems has at least \(p^{m - 2} \cdot p = p^{m - 1}\) solutions, and at most \(\size{G^{ab}} = p^{m}\). We conclude that \(\ch_{1, \phi} \in \{p^{m - 1}, p^{m}\}\).

	Next, suppose that \(\ch_{1, \phi} = p^{m}\). By \cref{lem:k1phiSMC(1mp)} we then know that \(\phi(x) = x^{Ap^{m - 1} + 1}\) for some \(A \in \Z\). Fix \(\overline{\chi}_{l} \in \Irr_{p}(G)\) and write \(\alpha(y) \equiv \beta p^{m - 1} + 1 \bmod p^{m}\) with \(\beta\) coprime with \(p\), since \(\alpha(y) \not \equiv 1 \bmod p^{m}\). Let \(\zeta\) be as before. Note that
	\[
		\zeta^{l p (Ap^{m - 1} + 1 - 1)} = \zeta^{l A p^{m}} = 1.
	\]
	Thus, by the proof of \cref{prop:kpphiSMC(1mp)}, this implies that \(\overline{\chi}_{l} \circ \phi = \overline{\chi}_{l}\). As \(\overline{\chi}_{l}\) was arbitrary, we conclude that \(\ch_{p, \phi} = p^{m - 2}(p - 1)\).
	
	Thus, we have proven that
	\[
		\SpecR(G) \subseteq
		\begin{cases}
			 \{p^{i} \mid 1 \leq i \leq m - 1\} \cup \{2p^{m - 1} - p^{m - 2}, p^{m} + p^{m - 1} - p^{m - 2}\}	&	\mbox{for \(p\) odd,}	\\
			 \{2^{i} \mid 2 \leq i \leq m - 1\} \cup \{2^{m} - 2^{m - 2}, 2^{m} + 2^{m - 2}\}	&	\mbox{for \(p = 2\).}	
		\end{cases}
	\]
	
	We are left with providing automorphisms of \(G\) realising the candidate-Reidemeister numbers on the right. For \(i \in \{0, \ldots, m - 1\}\), define
	\[
		\phi_{i}: G \to G: x \mapsto x^{p^{i} + 1}, y \mapsto y.
	\]
	It is readily verified that \(\phi_{i}\) preserves the relations of \(G\) and that \(\phi_{i}\) is surjective (and therefore injective) for all \(i \in \{0, \ldots, m - 1\}\) if \(p\) is odd, and for all \(i \in \{1, \ldots, m - 1\}\) if \(p = 2\). The map \(\ab{\phi}_{i}\) is given by \(\ab{\phi}_{i}(\bar{x}^{A}\bar{y}^{B}) = \bar{x}^{A(p^{i} + 1)}\bar{y}^{B}\), which has \(p^{i + 1}\) fixed points, showing that \(\ch_{1, \phi_{i}} = p^{i + 1}\). If \(i \leq m - 2\), then
	\[
		\overline{\chi}_{1}(\phi(x^{p})) = \overline{\chi}_{1}(x^{p^{i + 1} + p}) = p \zeta^{p^{i + 1} + p} \ne p \zeta^{p} = \overline{\chi}_{1}(x^{p}),
	\]
	showing that \(\overline{\chi}_{1}\) is not fixed by \(\phi_{i}\). Therefore, \(\ch_{p, \phi_{i}} = 0\) if \(i \leq m - 2\). If \(i = m - 1\), then the equality \(p \zeta^{p^{i + 1} + 1} = p \zeta^{p}\) holds and thus \(\ch_{p, \phi_{i}} = p^{m - 2}(p - 1)\). Consequently, we see that
	\[
		R(\phi_{i}) = \begin{cases}
			p^{i + 1}	&	\mbox{if } i \leq m - 2	\\
			p^{m} + p^{m - 2}(p - 1)	&	\mbox{if } i = m - 1.
		\end{cases}
	\]
	Finally, define \(\psi: G \to G: x \mapsto x^{p^{m - 1} + 1}y, y \mapsto y\). Again, \(\psi\) is a well-defined automorphism of \(G\). For \(\overline{\chi}_{l} \in \Irr_{p}(G)\) and \(e \in \{0, \ldots, p^{m - 1} - 1\}\), we see that
	\[
		\overline{\chi}_{l}(\psi(x^{pe})) = \overline{\chi}_{l}(x^{p^{m}e + pe}) = \overline{\chi}_{l}(x^{pe}).
	\]
	Thus, \(\overline{\chi}_{l}\) and \(\overline{\chi}_{l} \circ \psi\) match on \(Z(G)\). As \(\overline{\chi}_{l}\) is zero outside \(Z(G)\), it follows that \(\overline{\chi}_{l}\) is fixed by \(\psi\). Therefore, \(\ch_{p, \psi} = p^{m - 2}(p - 1)\), as \(\overline{\chi}_{l}\) was arbitrary. Now, on \(\ab{G}\) we see that \(\psi(\bar{x}^{A}\bar{y}^{B}) = \bar{x}^{A} \bar{y}^{B + A}\). Thus, \(\bar{x}^{A}\bar{y}^{B}\) is a fixed point if and only if \(A \equiv 0 \bmod p\). This implies that \(\ab{\psi}\) has \(p^{m - 2} \cdot p = p^{m - 1}\) fixed points, showing that
	\[
		R(\psi) = R(\ab{\psi}) + \ch_{p, \psi} = p^{m - 1} + p^{m - 2}(p - 1),
	\]
	ending the proof.
\end{proof}

\section*{Acknowledgments}
The author thanks Karel Dekimpe for his useful remarks and suggestions.

\printbibliography[heading=bibintoc]
\end{document}